\documentclass[reqno,12pt,letterpaper]{amsart}
\usepackage{amsmath,amssymb,amsthm,graphicx,url,bookmark,hyperref}

\newtheorem{theo}{Theorem}
\newtheorem{prop}{Proposition}[section]
\newtheorem{defi}{Definition}[section]
\numberwithin{equation}{section}

\DeclareMathOperator{\Real}{Re}
\DeclareMathOperator{\Imag}{Im}
\DeclareMathOperator{\loc}{loc}
\DeclareMathOperator{\comp}{comp}
\DeclareMathOperator{\const}{const}
\DeclareMathOperator{\supp}{supp}
\DeclareMathOperator{\Hol}{Hol}
\DeclareMathOperator{\Div}{div}
\DeclareMathOperator{\Vol}{Vol}

\DeclareMathOperator{\sgn}{sgn}

\title[Quasi-normal modes for Kerr--de Sitter]%
{Quasi-normal modes and exponential energy decay\\
for the Kerr-de Sitter black hole}
\author{Semyon Dyatlov}
\email{dyatlov@math.berkeley.edu}
\address{Department of Mathematics, Evans Hall, University of California,
Berkeley, CA 94720, USA}

\begin{document}

\begin{abstract}
We provide a rigorous definition of quasi-normal modes for a rotating
black hole.  They are given by the poles of a certain meromorphic
family of operators and agree with the heuristic definition in the
physics literature.  If the black hole rotates slowly enough, we show
that these poles form a discrete subset of $\mathbb C$. As an
application we prove that the local energy of linear waves in that
background decays exponentially once orthogonality to the zero
resonance is imposed.
\end{abstract}

\maketitle


Quasi-normal modes are the complex frequencies appearing in expansions
of waves; their real part corresponds to the rate of oscillation and
the nonpositive imaginary part, to the rate of decay.  According to the physics
literature~\cite{k-s} they are expected to appear in
gravitational waves caused by perturbations of black holes (for more
recent references and findings, see for example~\cite{b-c-s,k-z-1,k-z-2}).
In the mathematics literature they were studied by Bachelot and
Motet-Bachelot~\cite{b-1,b-2,b-3} and S\'a Barreto and
Zworski~\cite{sb-z}, who applied the methods of scattering theory and
semiclassical analysis to the case of a spherically symmetric black
hole.  Quasi-normal modes were described in~\cite{sb-z} as
\textbf{resonances}; that is, poles of the meromorphic continuation of
a certain family of operators; it was also proved that these poles
asymptotically lie on a lattice. This was further developed by Bony
and H\"afner in~\cite{b-h}, who established an expansion of the
solutions of the wave equation in terms of resonant states.
As a byproduct of this result, they obtained exponential decay
of local energy for Schwarzschild--de Sitter. Melrose, S\'a Barreto,
and Vasy~\cite{m-sb-v} have extended this result to more general manifolds
and more general initial data.

In this paper, we employ different methods to define quasi-normal
modes for the Kerr--de Sitter rotating black hole. As
in~\cite{sb-z} and~\cite{b-h}, we use the de Sitter model; physically,
this corresponds to a positive cosmological constant; mathematically,
it replaces asymptotically Euclidean spatial infinity with an asymptotically
hyperbolic one.
Let $P_g(\omega)$, $\omega\in \mathbb C$, be the stationary
d'Alembert--Beltrami operator of the Kerr--de Sitter metric (see
Section~1 for details). It acts on functions
on the space slice $M=(r_-,r_+)\times \mathbb S^2$.
We define quasi-normal modes as poles of a
certain (right) inverse $R_g(\omega)$ to $P_g(\omega)$. Because of the
cylindrical symmetry of the operator $P_g(\omega)$, it leaves
invariant the space $\mathcal D'_k$ of distributions with fixed
angular momentum $k\in \mathbb Z$ (with respect to the axis of
rotation); the inverse $R_g(\omega)$ on $\mathcal D'_k$ is constructed
by
\begin{theo}\label{t:r-g-omega-k}
Let $P_g(\omega,k)$ be the restriction of $P_g(\omega)$ to
$\mathcal D'_k$. Then there exists a family of operators
$$
R_g(\omega,k):
L^2_{\comp}(M)\cap \mathcal D'_k
\to H^2_{\loc}(M)\cap \mathcal D'_k
$$
meromorphic in $\omega\in \mathbb C$ with poles of finite rank and
such that $P_g(\omega,k)R_g(\omega,k)f=f$ for each $f\in
L^2_{\comp}(M)\cap \mathcal D'_k$.
\end{theo}
Since $R_g(\omega,k)$ is meromorphic, its poles, which we call
$k$-resonances, form a discrete set. One can then say that
$\omega\in \mathbb C$ is a resonance, or a quasi-normal mode, if
$\omega$ is a $k$-resonance for some $k\in \mathbb Z$.  However, it is
desirable to know that resonances form a discrete subset of $\mathbb
C$; that is, $k$-resonances for different $k$ do not accumulate near some
point. Also, one wants to construct the inverse $R_g(\omega)$ that
works for all values of $k$.
For $\delta_r>0$,\footnote{In this paper, the subscript in the constants
such as
$\delta_r$, $C_r$, $C_\theta$ does not mean that these constants
depend on the corresponding variables, such as $r$ or $\theta$; instead,
it indicates that they are related to these variables.} put
$$
K_r=(r_-+\delta_r,r_+-\delta_r),\
M_K=K_r\times \mathbb S^2,
$$
and let $1_{M_K}$ be the operator of multiplication
by the characteristic function of $M_K$ (which will, based on the context,
act $L^2(M_K)\to L^2(M)$ or $L^2(M)\to L^2(M_K)$).
Then we are able to construct $R_g(\omega)$ on $M_K$ for a slowly
rotating black hole:
\begin{theo}\label{t:r-g-omega}
Fix $\delta_r>0$. Then there exists $a_0>0$
such that if the rotation speed of the black hole satisfies $|a|<a_0$,
we have the following:

1. Every fixed compact set can only contain $k$-resonances for a
finite number of values of $k$.  Therefore, quasi-normal modes form a
discrete subset of $\mathbb C$.

2. The operators
$1_{M_K}R_g(\omega,k) 1_{M_K}$ define a family of operators
$$
R_g(\omega):L^2(M_K)\to H^2(M_K)
$$ 
such that $P_g(\omega)R_g(\omega)f=f$ on $M_K$ for
each $f\in L^2(M_K)$ and $R_g(\omega)$ is meromorphic
in $\omega\in \mathbb C$ with poles of finite rank.
\end{theo}
As stated in Theorem~\ref{t:r-g-omega}, the operator $R_g(\omega)$ acts
only on functions supported in a certain compact subset of the space
slice $M$ depending on how small $a$ is. This is due to the fact that the
operator $P_g(\omega)$ is not elliptic inside the two
ergospheres located near the endpoints $r=r_\pm$.  The result
above can then be viewed as a construction of $R_g(\omega)$ away from
the ergospheres. However, for fixed angular momentum we are able
to obtain certain boundary conditions on the elements in
the image of $R_g(\omega,k)$, as well as on resonant states:
\begin{theo}\label{t:global-outgoing} Let $\omega\in \mathbb C$.

1. Assume that $\omega$ is not a resonance. Take
$f\in L^2_{\comp}(M)\cap \mathcal D'_k$ for some $k\in \mathbb Z$
and put $u=R_g(\omega,k)f\in H^2_{\loc}(M)$.
Then $u$ is \textbf{outgoing} in the following sense: the functions  
$$
v_\pm(r,\theta,\varphi)=|r-r_\pm|^{iA_\pm^{-1}(1+\alpha)(r_\pm^2+a^2)\omega}
u(r,\theta,\varphi-A_\pm^{-1}(1+\alpha)a\ln|r-r_\pm|).
$$
are smooth near the event horizons $\{r_\pm\}\times \mathbb S^2$. 

2. Assume that $\omega$ is a resonance.
Then there exists a resonant state; i.e., a
nonzero solution $u\in C^\infty(M)$ to the
equation $P_g(\omega)u=0$ that is outgoing in the sense
of part 1.
\end{theo}
The outgoing condition can be reformulated as follows. Consider the
function $U=e^{-i\omega t}u$ on the spacetime $\mathbb R\times M$;
then $u$ is outgoing if and only if $U$ is smooth up to the event
horizons in the extension of the metric given by the Kerr-star
coordinates $(t^*,r,\theta,\varphi^*)$ discussed in Section~1. This
lets us establish a relation between the wave equation on Kerr-de
Sitter and the family of operators $R_g(\omega)$
(Proposition~\ref{l:our-resolvent-relevance}). Note that here we do
not follow earlier applications of scattering theory
(including~\cite{b-h}), where spectral theory and in particular
self-adjointness of $P_g$ are used to define $R_g(\omega)$ for
$\Imag\omega>0$ and relate it to solutions of the wave equation via
Stone's formula. In the situation of the present paper, due to the
lack of ellipticity of $P_g(\omega)$ inside the ergospheres, it is
doubtful that $P_g$ can be made into a self-adjoint operator;
therefore, we construct $R_g(\omega)$ directly using separation of
variables, cite the theory of hyperbolic equations (see Section~1) for
well-posedness of the Cauchy problem for the wave equation, and prove
Proposition~\ref{l:our-resolvent-relevance} without any reference to
spectral theory.

We now study the distribution of resonances in the slowly rotating
Kerr--de Sitter case. First, we establish absense of nonzero
resonances in the closed upper half-plane:
\begin{theo}\label{t:upper-half-plane}
Fix $\delta_r>0$. Then there exist constants $a_0$ and
$C$ such that if $|a|<a_0$, then:

1. There are no resonances in the upper half-plane and
$$
\|R_g(\omega)\|_{L^2(M_K)\to L^2(M_K)}
\leq {C\over |\Imag\omega|^2},\ \Imag\omega>0.
$$

2. There are no resonances $\omega\in \mathbb R\setminus 0$ and
$$
R_g(\omega)={i(1 \otimes 1)\over 4\pi (1+\alpha)(r_+^2+r_-^2+2a^2)\omega}+\Hol(\omega),
$$
where $\Hol$ stands for a family of operators holomorphic at zero.
\end{theo}

Next, we use the methods of~\cite{w-z} and the fact that the only
trapping in our situation is normally hyperbolic to get a resonance
free strip:
\begin{theo}\label{t:resonance-free-strip}
Fix $\delta_r>0$ and $s>0$. Then there exist $a_0>0$, $\nu_0>0$,
and $C$ such that for $|a|<a_0$,
$$
\|R_g(\omega)\|_{L^2(M_K)\to L^2(M_K)}\leq C|\omega|^s,\
|\Real\omega|\geq C,\
|\Imag\omega|\leq \nu_0.
$$
\end{theo}

Theorems~\ref{t:upper-half-plane} and~\ref{t:resonance-free-strip}, together with
the fact that resonances form a discrete set, imply that for $\nu_0$
small enough, zero is the only resonance in $\{\Imag\omega\geq-\nu_0\}$. This and the
presence of the global meromorphic continuation provide exponential decay of local energy:%
\footnote{Recently, the author has obtained stronger exponential
decay results
(\href{http://arxiv.org/abs/1010.5201}{\texttt{arXiv:1010.5201}}), as
well as a more precise description of resonances and a resonance
decomposition
(\href{http://arxiv.org/abs/1101.1260}{\texttt{arXiv:1101.1260}}); all
of these are based on the present paper.}
\begin{theo}\label{t:exponential-decay}
Let
$(r,t^*,\theta,\varphi^*)$ be the coordinates on the Kerr--de Sitter
background introduced in Section~1.  Fix $\delta_r>0$ and $s'>0$ and
assume that $a$ is small enough.  Let $u$ be a solution to the wave
equation $\Box_g u=0$ with initial data
\begin{equation}\label{e:exponential-decay-ivp}
\begin{gathered}
u|_{t^*=0}=f_0\in H^{3/2+s'}(M)\cap \mathcal E'(M_K),\\
\partial_{t^*}u|_{t^*=0}=f_1\in H^{1/2+s'}(M)\cap \mathcal E'(M_K).
\end{gathered}
\end{equation}
Also, define the constant
$$
u_0={1+\alpha\over 4\pi (r_+^2+r_-^2+2a^2)}\int_{t^*=0}*(du).
$$
Here $*$ denotes the Hodge star operator for the metric $g$ (see Section~1). Then
$$
\|u(t^*,\cdot)-u_0\|_{L^2(M_K)}\leq Ce^{-\nu' t^*}(\|f_0\|_{H^{3/2+s'}(M_K)}+\|f_1\|_{H^{1/2+s'}(M_K)}),\
t^*>0,
$$
for certain constants $C$ and $\nu'$ independent of $u$.
\end{theo}

For the Kerr metric, the local energy decay is polynomial as shown by Tataru
and Tohaneanu~\cite{ta,t-t},
see also the lecture notes by Dafermos and Rodnianski~\cite{d-r} and the references below.

\noindent\textbf{Outline of the proof.}
The starting point of the construction of
$R_g(\omega)$ is the separation of variables
introduced by Teukolsky in~\cite{t}. The separation of variables
techniques and the related symmetries
have been used in many papers, including~\cite{a-b}, \cite{b-s}, 
\cite{b-h}, \cite{dss,dss1}, \cite{fksy,fksy:erratum},
\cite{sb-z}, \cite{to};
however, these mostly consider the case of zero cosmological
constant, where other difficulties occur at zero energy
and a global meromorphic continuation of the type presented 
here is unlikely. In our case, since the metric is
invariant under axial rotation, it is enough to construct the
operators $R_g(\omega,k)$ and study their behavior for large
$k$. The operator $P_g(\omega,k)$ is next decomposed into the sum of
two ordinary differential operators, $P_r$ and $P_\theta$
(see~\eqref{e:p-r-theta}). The separation of variables is discussed in
Section~1; the same section contains the derivation of
Theorem~\ref{t:exponential-decay} from the other theorems
by the complex contour deformation method. 

In the Schwarzchild--de Sitter case, $P_\theta$ is just the
Laplace--Beltrami operator on the round sphere and one can use
spherical harmonics to reduce the problem to studying the operator
$P_r+\lambda$ for large $\lambda$. In
the case $a\neq 0$, however, the operator $P_\theta$ is
$\omega$-dependent; what is more, it is no longer self-adjoint unless
$\omega\in \mathbb R$.  This raises two problems with the standard
implementation of separation of variables, namely decomposing $L^2$
into a direct sum of the eigenspaces of $P_\theta$. Firstly, since
$P_\theta$ is not self-adjoint, we cannot automatically guarantee
existence of a complete system of eigenfunctions and the corresponding
eigenspaces need not be orthogonal. Secondly, the eigenvalues of
$P_\theta$ are functions of $\omega$, and meromorphy of $R_g(\omega)$
is nontrivial to show when two of these eigenvalues
coincide. Therefore, instead of using the eigenspace decomposition, we
write $R_g(\omega)$ as a certain contour
integral~\eqref{e:r-omega-int} in the complex plane; the proof of
meromorphy of this integral is based on Weierstrass preparation
theorem. This is described in Section~2.

In Section~3, we use the separation of variables procedure to reduce
Theorems~\ref{t:r-g-omega-k}--\ref{t:upper-half-plane} to certain
facts about the radial resolvent $R_r$ (Proposition~\ref{l:radial}).
For fixed $\omega,\lambda,k$, where $\lambda\in \mathbb C$ is the
separation constant, $R_r$ is constructed in Section~4 using the
methods of one-dimensional scattering theory. Indeed, the radial
operator $P_r$, after the Regge--Wheeler change of
variables~\eqref{e:regge-wheeler}, is equivalent to the Schr\"odinger
operator $P_x=D_x^2+V_x(x)$ for a certain potential $V_x$
\eqref{e:p-x}. (Here $x=\pm\infty$ correspond
to the event horizons.) This does not, however, provide estimates on $R_r$ that
are uniform as $\omega,\lambda,k$ go to infinity.

The main difficulty then is proving a uniform resolvent
estimate (see~\eqref{e:radial-est}),
valid for large $\lambda$ and $\Real\lambda\gg
|\Imag\lambda|+|\omega|^2+|ak|^2$, which in particular guarantees the
convergence of the integral~\eqref{e:r-omega-int} and
Theorem~\ref{t:r-g-omega}.  A complication arises from the fact that
$V_x(\pm\infty)=-\omega_\pm ^2$, where $\omega_\pm$ are proportional
to $(r_\pm^2+a^2)\omega-ak$. No matter how large $\omega$ is, one can
always choose $k$ so that one of $\omega_\pm$ is small, making it
impossible to use standard complex scaling%
\footnote{Complex scaling originated in mathematical physics with
the work of Aguilar--Combes~\cite{a-c}, Balslev--Combes, and Simon.  It has
become a standard tool in chemistry for computing resonances. A
microlocal approach has been developed by Helffer--Sj\"ostrand, and a
more geometric version by Sj\"ostrand--Zworski~\cite{s-z}~--- see that
paper for pointers to the literature. Complex scaling was reborn
in numerical analysis in 1994 as the method of ``perfectly matched
layers'' (see~\cite{ber}).  A nice application of the method of complex scaling to the
Schwarzschild--de Sitter case is provided in \cite{d}.}, in the case
$\omega=o(k)$, due to the lack of ellipticity of the rescaled operator
at infinity. To avoid this issue, we use the analyticity of $V_x$ and
semiclassical analysis to get certain control on outgoing solutions at
two distant, but fixed, points (Proposition~\ref{l:vertical}), and
then an integration by parts argument to get an $L^2$ bound between
these two points. This is discussed in Section~6.

Finally, Section~7 contains the proof of
Theorem~\ref{t:resonance-free-strip}.  We first use the results of
Sections~2--6 to reduce the problem to scattering for the
Schr\"odinger operator $P_x$ in the regime $\lambda=O(\omega^2)$,
$k=O(\omega)$ (Proposition~\ref{l:resonance-free-radial}).  In this
case, we apply complex scaling to deform $P_x$ near $x=\pm\infty$ to an elliptic operator
(Proposition~\ref{l:complex-scaling}).  We then analyse the
corresponding classical flow; it is either nontrapping at zero energy,
in which case the usual escape function construction (as in, for
example,~\cite{m}) applies (Proposition~\ref{l:radial-nontrapping}),
or has a unique maximum. In the latter case we use the methods
of~\cite{w-z} designed to handle more general normally hyperbolic
trapped sets and based on commutator estimates in a slightly exotic
microlocal calculus. The argument of~\cite{w-z} has to be modified to
use complex scaling instead of an absorbing potential near infinity
(see also~\cite[Theorem~2]{w-z}).

It should be noted that, unlike~\cite{b-h} or~\cite{sb-z}, the construction
of $R_g(\omega)$ in the present
paper does not use the theorem of Mazzeo--Melrose~\cite{m-m} on the
meromorphic continuation of the resolvent on spaces with
asymptotically constant negative curvature (see also~\cite{g}).
In~\cite{b-h} and~\cite{sb-z}, this theorem had to be applied to prove
the existence of the meromorphic continuation of the resolvent for
$\omega$ in a fixed neighborhood of zero where complex scaling could
not be implemented.

\noindent\textbf{Remark.} The results of this paper also apply if the wave
equation is replaced by the Klein--Gordon equation~\cite{k}
$$
(\Box_g+m^2)u=0,
$$
where $m>0$ is a fixed constant. The corresponding stationary operator
is $P_g(\omega)+m^2\rho^2$; when restricted to the space $\mathcal D'_k$, it is the
sum of the two operators (see Section~1)
$$
\begin{gathered}
P_r(\omega,k;m)=P_r(\omega,k)+m^2r^2,\
P_\theta(\omega;m)=P_\theta(\omega)+m^2a^2\cos^2\theta.
\end{gathered}
$$
The proofs in this paper all go through in this case as well. 
In particular, the rescaled radial operator $P_x$ introduced in~\eqref{e:p-x}
is a Schr\"odinger operator with the potential
$$
V_x(x;\omega,\lambda,k;m)=(\lambda+m^2r^2)\Delta_r-(1+\alpha)^2((r^2+a^2)\omega-ak)^2.
$$
Since $V_x(\pm\infty)$ is still equal to $-\omega_\pm^2$ with $\omega_\pm$ defined
in~\eqref{e:omega-pm}, the radial resolvent can be defined as a meromorphic
family of operators on the entire complex plane. Also, the term $m^2r^2\Delta_r$ 
in the operator $P_x$ becomes of order $O(h^2)$ under
the semiclassical rescaling and thus does not affect the arguments in Sections~6 and~7.

The only difference in the Klein--Gordon case is the absense of the resonance
at zero: $1$ is no longer an outgoing solution to the equation
$P_x u=0$ for $\omega=k=\lambda=0$.
Therefore, there is no $u_0$ term in Theorem~\ref{t:exponential-decay},
and all solutions to~\eqref{e:exponential-decay-ivp} decay exponentially
in the compact set $M_K$.

\section{Kerr--de Sitter metric}

The Kerr--de Sitter metric is given by the formulas \cite{c}
$$
\begin{gathered}
g=-\rho^2\Big({dr^2\over \Delta_r}+{d\theta^2\over\Delta_\theta}\Big)\\
-{\Delta_\theta\sin^2\theta\over (1+\alpha)^2\rho^2}
(a\,dt-(r^2+a^2)\,d\varphi)^2\\
+{\Delta_r\over (1+\alpha)^2\rho^2}
(dt-a\sin^2\theta\,d\varphi)^2.
\end{gathered}
$$
Here $M_0$ is the mass of the black hole, $\Lambda$ is the cosmological
constant (both of which we assume to be fixed throughout the paper),
and $a$ is the angular momentum (which we assume to be bounded by some constant,
and which is required to be small by most of our theorems);
$$
\begin{gathered}
\Delta_r=(r^2+a^2)\Big(1-{\Lambda r^2\over 3}\Big)-2M_0r,\\
\Delta_\theta=1+\alpha\cos^2\theta,\\
\rho^2=r^2+a^2\cos^2\theta,\
\alpha={\Lambda a^2\over 3}.
\end{gathered}
$$
We also put
$$
A_\pm=\mp\partial_r\Delta_r(r_\pm)>0.
$$
The metric is defined for $\Delta_r>0$; we assume that this happens on
an open interval $0<r_-<r<r_+<\infty$. (For $a=0$, this is true when
$9\Lambda M_0^2<1$; it remains true if we take $a$ small enough.) The
variables $\theta\in [0,\pi]$ and $\varphi\in \mathbb R/2\pi \mathbb Z$
are the spherical coordinates on the sphere $\mathbb S^2$. We define
the space slice $M=(r_-,r_+)\times \mathbb S^2$; then the Kerr--de Sitter
metric is defined on the spacetime $\mathbb R\times M$. 


The d'Alembert--Beltrami operator of $g$ is given by
$$
\begin{gathered} 
\Box_g={1\over\rho^2}D_r(\Delta_rD_r)
+{1\over\rho^2\sin\theta}D_\theta(\Delta_\theta\sin\theta D_\theta)\\
+{(1+\alpha)^2\over\rho^2\Delta_\theta\sin^2\theta}
(a\sin^2\theta D_t+D_\varphi)^2\\
-{(1+\alpha)^2\over\rho^2\Delta_r}
((r^2+a^2) D_t+aD_\varphi)^2.
\end{gathered} 
$$
(Henceforth we denote $D={1\over i}\partial$.)
The volume form is
$$
d\Vol={\rho^2\sin\theta\over (1+\alpha)^2}\,dtdrd\theta d\varphi.
$$
If we replace $D_t$ by a number $-\omega\in \mathbb C$, then the
operator $\Box_g$ becomes equal to $P_g(\omega)/\rho^2$, where
$P_g(\omega)$ is the following differential operator on
$M$:
\begin{equation}\label{e:p-g-omega}
\begin{gathered}
P_g(\omega)=D_r(\Delta_r D_r)-{(1+\alpha)^2\over\Delta_r}((r^2+a^2)\omega-aD_\varphi)^2\\
+{1\over\sin\theta}D_\theta(\Delta_\theta\sin\theta D_\theta)
+{(1+\alpha)^2\over\Delta_\theta\sin^2\theta}(a\omega\sin^2\theta-D_\varphi)^2.
\end{gathered} 
\end{equation}

We now introduce the separation of variables for the operator $P_g(\omega)$.
We start with taking Fourier series in the variable $\varphi$. For every
$k\in \mathbb Z$, define the space
\begin{equation}\label{e:d-prime-k}
\mathcal D'_k=\{u\in \mathcal D'\mid (D_\varphi-k)u=0\}.
\end{equation}
This space can be considered as a subspace of $\mathcal D'(M)$
or of $\mathcal D'(\mathbb S^2)$ alone, and 
$$
L^2(M)=\bigoplus_{k\in \mathbb Z}(L^2(M)\cap \mathcal D'_k);
$$
the right-hand side is the Hilbert sum of a family of closed mutually orthogonal subspaces.

Let $P_g(\omega,k)$ be the restriction of $P_g(\omega)$ to $\mathcal
D'_k$.  Then we can write
$$
P_g(\omega,k)=P_r(\omega,k)+P_\theta(\omega)|_{\mathcal D'_k},
$$
where
\begin{equation}\label{e:p-r-theta}
\begin{gathered}
P_r(\omega,k)=D_r(\Delta_rD_r)-{(1+\alpha)^2\over\Delta_r}((r^2+a^2)\omega-ak)^2,\\
P_\theta(\omega)={1\over\sin\theta}D_\theta(\Delta_\theta\sin\theta D_\theta)
+{(1+\alpha)^2\over \Delta_\theta\sin^2\theta}(a\omega\sin^2\theta-D_\varphi)^2
\end{gathered}
\end{equation}
are differential operators in $r$ and $(\theta,\varphi)$, respectively. 

Next, we introduce a modification of the Kerr-star coordinates
(see~\cite[Section~5.1]{d-r}).
Following~\cite{t-t}, we remove the singularities at $r=r_\pm$ by making
the change of variables $(t,r,\theta,\varphi)\to (t^*,r,\theta,\varphi^*)$,
where
$$
t^*=t-F_t(r),\
\varphi^*=\varphi-F_\varphi(r).
$$
Note that $\partial_{t^*}=\partial_t$ and $\partial_{\varphi^*}=\partial_\varphi$.
In the new coordinates, the metric becomes
$$
\begin{gathered}
g=-\rho^2\Big({dr^2\over \Delta_r}+{d\theta^2\over\Delta_\theta}\Big)\\
-{\Delta_\theta\sin^2\theta\over (1+\alpha)^2\rho^2}
[a\,dt^*-(r^2+a^2)\,d\varphi^*+(aF'_t(r)-(r^2+a^2)F'_\varphi(r))dr]^2\\
+{\Delta_r\over (1+\alpha)^2\rho^2}
[dt^*-a\sin^2\theta\,d\varphi^*+(F'_t(r)-a\sin^2\theta F'_\varphi(r))dr]^2.
\end{gathered}
$$
The functions $F_t$ and $F_\varphi$ are required to be smooth on $(r_-,r_+)$ and
satisfy the following conditions:
\begin{itemize}
\item $F_t(r)=F_\varphi(r)=0$ for $r\in K_r=[r_-+\delta_r,r_+-\delta_r]$;
\item $F'_t(r)=\pm{(1+\alpha)(r^2+a^2)/\Delta_r}+F_{t_\pm}(r)$
and $F'_\varphi(r)=\pm{(1+\alpha)a/\Delta_r}+F_{\varphi\pm}(r)$,
where $F_{t\pm}$ and $F_{\varphi\pm}$ are smooth at $r=r_\pm$, respectively;
\item for some ($a$-independent) constant $C$ and all $r\in (r_-,r_+)$,
$$
{(1+\alpha)^2 (r^2+a^2)^2\over\Delta_r}-\Delta_rF'_t(r)^2-(1+\alpha)^2a^2\geq {1\over C}>0.
$$
\end{itemize}
Under these conditions, the metric $g$ in the new coordinates is smooth
up to the event horizons $r=r_\pm$ and the space slices
$$
M_{t_0}=\{t^*=t_0=\const\}\cap (\mathbb R\times M),\ t_0\in \mathbb R,
$$
are space-like. Let $\nu_t$ be the
time-like normal vector field to these surfaces, chosen so that
$g(\nu_t,\nu_t)=1$ and $\langle dt^*,\nu_t\rangle>0$. 

We now establish a basic energy estimate for the wave equation in our setting.
Let $u$ be a real-valued function smooth in the coordinates $(t^*,r,\theta,\varphi^*)$
up to the event horizons. Define the vector field $T(du)$ by
$$
T(du)=\partial_t u\nabla_g u
-{1\over 2}g(du,du)\nu_t.
$$
Since $\nu_t$ is timelike, the expression $g(T(du),\nu_t)$
is a positive definite quadratic form in $du$.
For $t_0\in \mathbb R$, define $E(t_0)(du)$ as the integral
of this quadratic form over the space slice $M_{t_0}$ with the
volume form induced by the metric.
\begin{prop}\label{l:crude-energy-estimate} Take $t_1<t_2$ and let
$$
\Omega=\{t_1\leq t^*\leq t_2\}\times M.
$$
Assume that $u$ is smooth in $\Omega$ up to its boundary
and solves the wave equation $\Box_g u=0$ in this region. Then
$$
E(t_2)(du)\leq e^{C_e(t_2-t_1)}E(t_1)(du)
$$
for some constant $C_e$ independent of $t_1$ and $t_2$.
\end{prop}
\begin{proof} We use the method of~\cite[Proposition~2.8.1]{tay}.
We apply the divergence theorem to the vector field $T(du)$ on the domain
$\Omega$. The integrals over $M_{t_1}$ and $M_{t_2}$ will be equal to
$-E(t_1)$ and $E(t_2)$. The restriction of the metric to tangent spaces
of the event horizons is nonpositive and the field $\nu_t$ is pointing
outside of $\Omega$ at $r=r_\pm$; therefore, the integrals over the event
horizons will be nonnegative. Finally, since $\Box_g u=0$, one can prove
that $\Div T(du)$ is quadratic in $du$ and thus
$$
|\Div T(du)|\leq C g(T(du),\nu_t).
$$
Therefore, the divergence theorem gives
$$
E(t_2)-E(t_1)\leq C\int_{t_1}^{t_2} E(t_0)\,dt_0.
$$
It remains to use Gronwall's inequality.
\end{proof}
The geometric configuration of $\{t^*=t_1\}$, $\{t^*=t_2\}$,
$\{r=r_\pm\}$, and $\nu_t$ with respect to the Lorentzian metric $g$ used in
Proposition~\ref{l:crude-energy-estimate}, combined with the theory of
hyperbolic equations (see~\cite[Proposition~3.1.1]{d-r},
\cite[Theorem~23.2.4]{ho3}, or~\cite[Sections~2.8 and~7.7]{t}),
makes it possible to prove
that for each $f_0\in H^1(M)$,
$f_1\in L^2(M)$, there exists a unique solution
$$
u(t^*,\cdot)\in C([0,\infty);H^1(M))\cap C^1([0,\infty);L^2(M))
$$
to the initial value problem
\begin{equation}\label{e:wave-ivp}
\Box_g u=0,\
u|_{t^*=0}=f_0,\
\partial_{t^*}u|_{t^*=0}=f_1.
\end{equation}

We are now ready to prove Theorem~\ref{t:exponential-decay}.  Fix
$\delta_r>0$ and assume that $a$ is chosen small enough so that
Theorems~\ref{t:r-g-omega}--\ref{t:resonance-free-strip} hold.  Assume
that $s'>0$ and $u$ is the solution to~\eqref{e:wave-ivp} with $f_0\in
H^{3/2+s'}\cap \mathcal E'(M'_K)$ and $f_1\in H^{1/2+s'}\cap \mathcal
E'(M'_K)$, where $M'_K$ is fixed and compactly contained in $M_K$. By
finite propagation speed (see~\cite[Theorem~2.6.1
and~Section~2.8]{t}), there exists a function $\chi(t)\in
C^\infty(0,\infty)$ independent of $u$ and such that $\chi(t^*)=1$ for
$t^*>1$, and for $t^*\in\supp (1-\chi)$, $\supp u(t^*,\cdot)\subset
M_K$. By Proposition~\ref{l:crude-energy-estimate}, we can define the
Fourier-Laplace transform
$$
\widehat{\chi u}(\omega)=\int e^{it^*\omega}\chi(t^*)u(t^*,\cdot)\,dt^*\in H^{3/2+s'}(M),\
\Imag\omega>C_e.
$$
Put $f=\rho^2\Box_g(\chi u)=\rho^2[\Box_g,\chi]u$; then 
$$
f\in H^{1/2+s'}_{\comp}(\mathbb R;L^2(M)\cap \mathcal E'(M_K)).
$$ 
Therefore, one can define the Fourier-Laplace transform
$\hat f(\omega)\in L^2\cap \mathcal E'(M_K)$ for all $\omega\in \mathbb C$, and we have the estimate
$$
\int \langle\omega\rangle^{2s'+1}\|\hat f(\omega)\|_{L^2(M)}^2\,d\omega\leq C(\|f_0\|_{H^{3/2+s'}}^2
+\|f_1\|_{H^{1/2+s'}}^2).
$$
where integration is performed over the line $\{\Imag\omega=\nu=\const\}$
with $\nu$ bounded.

\begin{prop}\label{l:our-resolvent-relevance}
We have for $\Imag\omega>C_e$,
$$
\widehat{\chi u}(\omega)|_{M_K}=R_g(\omega)\hat f(\omega).
$$
\end{prop}
\begin{proof}
Without loss of generality, we may assume that $u\in C^\infty\cap \mathcal D'_k$
for some $k\in \mathbb Z$; then $R_g(\omega)\hat f(\omega)$ can be defined
on the whole $M$ by Theorem~\ref{t:r-g-omega-k}.
Fix $\omega$ and put
$$
\Phi(\omega)=e^{i\omega F_t(r)}\widehat{\chi u}(\omega)-R_g(\omega,k)\hat f(\omega)\in C^\infty(M).
$$
Since $\rho^2\Box_g(\chi u)=f$, we have
$$
P_g(\omega)(e^{i\omega F_t(r)}\widehat{\chi u}(\omega))=\hat f(\omega);
$$
therefore, $P_g(\omega)\Phi(\omega)=0$.  Note also that $\Phi$ is
smooth inside $M$ because of ellipticity of the operator $P_g(\omega)$
on $\mathcal D'_k$ (see~\cite[Section~7.4]{t} and the last step of the
proof of Theorem~\ref{t:r-g-omega-k}).  Now, if we put
$$
U(t,\cdot)=e^{-it\omega}\Phi(\omega)(\cdot),
$$
then $\Box_g U=0$ inside $M_S$. However, by Theorem~\ref{t:global-outgoing},
$U$ is smooth in the $(r,t^*,\theta,\varphi^*)$ coordinates up to the event horizons
and its energy grows in time faster than allowed by Proposition~\ref{l:crude-energy-estimate};
therefore, $\Phi=0$.
\end{proof}

We now restrict our attention to the compact $M_K$, where in particular $t=t^*$ and $\varphi=\varphi^*$.
By the Fourier Inversion Formula, for $t>1$ and $\nu>C_e$,
$$
u(t)|_{M_K}=(2\pi)^{-1}\int e^{-it(\omega+i\nu)}R_g(\omega+i\nu)\hat f(\omega+i\nu)\,d\omega.
$$
Fix positive $s<s'$. By Theorems~\ref{t:upper-half-plane} and~\ref{t:resonance-free-strip},
there exists $\nu_0>0$ such that zero is the only resonance with $\Imag\omega\geq -\nu_0$.
Using the estimates in these theorems, we can deform the contour of integration above
to the one with $\nu=-\nu_0$. Indeed, by a density argument we may assume that $u\in C^\infty$,
and in this case, $\hat f(\omega)$ is rapidly decreasing as $\Real\omega\to\infty$ for $\Imag\omega$ fixed.
We then get
\begin{equation}\label{e:wave-equation-last-integral}
\begin{gathered}
u(t)|_{K_r}={1+\alpha\over 4\pi (r_+^2+r_-^2+2a^2)}(\hat f(0),1)_{L^2(K_r)}
\\
+(2\pi)^{-1}e^{-\nu_0 t}\int e^{-it\omega}R_g(\omega-i\nu_0)\hat f(\omega-i\nu_0)\,d\omega.
\end{gathered}
\end{equation}
We find a representation of the first term above in terms of the initial data for $u$ at time zero.
We have
$$
(\hat f(0),1)_{L^2(K_r)}=\int_{M_K\times \mathbb R} \Box_g(\chi u)\,d\Vol.
$$
Here $d\Vol$ is the volume form induced by $g$.
Integrating by parts, we get
\begin{equation}\label{e:zero-resonance-contribution}
\int_{M_K\times \mathbb R} \Box_g(\chi u)\,d\Vol
=-\int_{t\geq 0}\Box_g((1-\chi) u)\,d\Vol
=\int_{t=0} *(du).
\end{equation}
Here $*$ is the Hodge star operator induced by the metric $g$, with the orientation
on $M$ and $\mathbb R\times M$ chosen so that $*(dt)$ is positively oriented
on $\{t=0\}$.

Finally, the $L^2$ norm of the integral term in~\eqref{e:wave-equation-last-integral} can be estimated by
$$
\begin{gathered}
Ce^{-\nu_0 t}\int \langle \omega\rangle^{s-s'-1/2} \|\langle\omega\rangle^{s'+1/2}\hat f(\omega-i\nu_0)\|_{L^2(K_r)}\,d\omega\\
\leq Ce^{-\nu_0 t}\|\langle\omega\rangle^{s'+1/2}\hat f(\omega-i\nu_0)\|_{L^2_\omega(\mathbb R)L^2(K_r)}\\
\leq Ce^{-\nu_0 t}(\|f_0\|_{H^{s'+3/2}}+\|f_1\|_{H^{s'+1/2}}),
\end{gathered}
$$
since $\langle\omega\rangle^{s-s'-1/2}\in L^2$. This proves Theorem~\ref{t:exponential-decay}.

\noindent\textbf{Remark.} In the original coordinates, $(t,r,\theta,\varphi)$,
the equation $\Box_g u=0$ has two solutions depending only on the time variable,
namely, $u=1$ and $u=t$. Even though Theorem~\ref{t:exponential-decay} does not
apply to these solutions because we only construct the family of operators $R_g(\omega)$
acting on functions
on the compact set $M_K$, it is still interesting to see where our argument fails if $R_g(\omega)$
were well-defined on the whole $M$. The key fact is that our Cauchy problem is formulated
in the $t^*$ variable. Then, for $u=t$ the function $f_0=u|_{t^*=0}$ behaves like
$\log|r-r_\pm|$ near the event horizons and thus does not lie in the energy space $H^1$.
As for $u=1$, our theorem gives the correct form of the contribution of the zero resonance,
namely, a constant; however, the value of this constant cannot be given by the integral
of $*(du)$ over ${t^*=0}$, as $du=0$. This discrepancy is explained if we look closer
at the last equation in~\eqref{e:zero-resonance-contribution}; while integrating by parts,
we will get a nonzero term coming from the integral of $*d(\chi(t^*))$ over the event horizons.

\section{Separation of variables in an abstract setting}

In this section, we construct inverses for certain families of
operators with separating variables. Since the method described below
can potentially be applied to other situations, we develop it
abstractly, without any reference to the operators of our problem.
Similar constructions have been used in other settings by
Ben-Artzi--Devinatz~\cite{b-d} and Mazzeo--Vasy~\cite[Section~2]{m-v}.

First, let us consider a differential operator
$$
P(\omega)=P_1(\omega)+P_2(\omega)
$$
in the variables $(x_1,x_2)$, where $P_1(\omega)$ is a differential
operator in the variable $x_1$ and $P_2(\omega)$ is a differential
operator in the variable $x_2$; $\omega$ is a complex parameter. If we
take $\mathcal H_1$ and $\mathcal H_2$ to be certain $L^2$ spaces in
the variables $x_1$ and $x_2$, respectively, then the corresponding
$L^2$ space in the variables $(x_1,x_2)$ is their Hilbert tensor
product $\mathcal H=\mathcal H_1 \otimes \mathcal H_2$.  Recall that
for any two bounded operators $A_1$ and $A_2$ on $\mathcal H_1$ and
$\mathcal H_2$, respectively, their tensor product $A_1 \otimes A_2$
is a bounded operator on $\mathcal H$ and
$$
\|A_1 \otimes A_2\|=\|A_1\|\cdot \|A_2\|.
$$
The operator $P$ is now written on $\mathcal H$ as
$$
P(\omega)=P_1(\omega) \otimes 1_{\mathcal H_2}+1_{\mathcal H_1} \otimes P_2(\omega).
$$

We now wish to construct an inverse to $P(\omega)$.  The method used
is an infinite-dimensional generalization of the following elementary
\begin{prop}\label{l:finite-sep-mer}
Assume that $A$ and $B$ are two (finite-dimensional) matrices and that
the matrix $A \otimes 1+1 \otimes B$ is invertible. (That is, no
eigenvalue of $A$ is the negative of an eigenvalue of $B$.) For
$\lambda\in \mathbb C$, let $R_A(\lambda)=(A+\lambda)^{-1}$ and
$R_B(\lambda)=(B-\lambda)^{-1}$. Take $\gamma$ to be a bounded simple
closed contour in the complex plane such that all poles of $R_A$ lie
outside of $\gamma$, but all poles of $R_B$ lie inside $\gamma$; we
assume that $\gamma$ is oriented in the clockwise direction. Then
$$
(A \otimes 1+1 \otimes B)^{-1}={1\over 2\pi i}\int_{\gamma} R_A(\lambda) \otimes R_B(\lambda)\,d\lambda.
$$
\end{prop}

The starting point of the method are the inverses%
\footnote{In this section, we do not use the fact that
$R_j(\omega,\lambda)=(P_j(\omega)\pm\lambda)^{-1}$, neither do we
prove that $R(\omega)=P(\omega)^{-1}$. This step will be done in our
particular case in the proof of Theorem~\ref{t:r-g-omega-k} in the
next section; in fact, $R_1$ will only be a right inverse to
$P_1+\lambda$. Until then, we merely establish properties of
$R(\omega)$ defined by~\eqref{e:r-omega-int} below.}
$$
\begin{gathered}
R_1(\omega,\lambda)=(P_1(\omega)+\lambda)^{-1},\
R_2(\omega,\lambda)=(P_2(\omega)-\lambda)^{-1}
\end{gathered}
$$
defined for $\lambda\in \mathbb C$. These inverses depend on two complex
variables, and we need to specify
their behavior near the singular points:
\begin{defi}\label{d:mer2}
Let $\mathcal X$ be any Banach space, and let $W$ be a domain in
$\mathbb C^2$.  We say that $T(\omega,\lambda)$ is an
($\omega$-nondegenerate) meromorphic map $W\to \mathcal X$ if:
\begin{enumerate}
\item[(1)] $T(\omega,\lambda)$ is a (norm) holomorphic function of two complex variables
with values in $\mathcal X$ for $(\omega,\lambda)\not\in Z$, where $Z$
is a closed subset of $W$, called the \textbf{divisor} of $T$,
\item[(2)] for each $(\omega_0,\lambda_0)\in Z$, we can write
$T(\omega,\lambda)=S(\omega,\lambda)/X(\omega,\lambda)$
near $(\omega_0,\lambda_0)$, where $S$ is holomorphic with values in
$\mathcal X$ and $X$ is a holomorphic function of two variables (with
values in $\mathbb C$) such that:
\begin{itemize}
\item for each $\omega$ close to $\omega_0$, there exists $\lambda$ such that
$X(\omega,\lambda)\neq 0$, and
\item the divisor of $T$ is given by $\{X=0\}$ near $(\omega_0,\lambda_0)$.
\end{itemize}
\end{enumerate}  
\end{defi}

Note that the definition above is stronger than the standard
definition of meromorphy and it is not symmetric in $\omega$ and
$\lambda$. Henceforth we will use this definition when talking about
meromorphic families of operators of two complex variables. It is
clear that any derivative (in $\omega$ and/or $\lambda$) of a
meromorphic family is again meromorphic. Moreover, if
$T(\omega,\lambda)$ is meromorphic and we fix $\omega$, then $T$ is a
meromorphic family in $\lambda$.

If $\mathcal X$ is the space of all bounded operators on some Hilbert
space (equipped with the operator norm), then it makes sense to talk
about having poles of finite rank:
\begin{defi}\label{d:2-mer-fin}
Let $\mathcal H$ be a Hilbert space and let $T(\omega,\lambda)$ be a
meromorphic family of operators on $\mathcal H$ in the sense of
Definition~\ref{d:mer2}. For $(\omega_0,\lambda_0)$ in the
divisor of $T$, consider the decomposition
$$
T(\omega_0,\lambda)=T_H(\lambda)+\sum_{j=1}^N {T_j\over (\lambda-\lambda_0)^j}.
$$
Here $T_H$ is holomorphic near $\lambda_0$ and $T_j$ are some
operators. We say that $T$ has \textbf{poles of finite rank} if every
operator $T_j$ in the above decomposition of every $\omega$-derivative
of $T$ near every point in the divisor is finite-dimensional.
\end{defi}

One can construct meromorphic families of operators with poles of
finite rank by using the following generalization of Analytic Fredholm
Theory:
\begin{prop}\label{l:aft-2}
Assume that $T(\omega,\lambda):\mathcal H_1\to \mathcal H_2$,
$(\omega,\lambda)\in \mathbb C^2$, is a
holomorphic family of Fredholm operators, where $\mathcal H_1$ and $\mathcal H_2$
are some Hilbert spaces. Moreover, assume that for each $\omega$, there exists
$\lambda$ such that the operator $T(\omega,\lambda)$ is
invertible. Then $T(\omega,\lambda)^{-1}$ is a meromorphic family
of operators $\mathcal H_2\to \mathcal H_1$ with poles of finite rank. (The divisor
is the set of all points where $T$ is not invertible.)
\end{prop} 
\begin{proof}
We can use the proof of the standard Analytic Fredholm Theory via Grushin problems,
see for example \cite[Theorem C.3]{e-z}.
\end{proof}

We now go back to constructing the inverse to $P(\omega)$. We assume that
\begin{enumerate}
\item[(A)]
$R_j(\omega,\lambda)$, $j=1,2$, are two
families of bounded operators on $\mathcal H_j$
with poles of finite rank. Here $\omega$ lies in
a domain $\Omega\subset \mathbb C$ and
$\lambda\in \mathbb C$.
\end{enumerate}

We want to integrate the tensor product $R_1\otimes R_2$ in $\lambda$
over a contour $\gamma$ that separates the sets of poles of
$R_1(\omega,\cdot)$ and $R_2(\omega,\cdot)$.  Let $Z_j$ be the divisor
of $R_j$.  We call a point $\omega$ \textbf{regular} if the sets
$Z_1(\omega)$ and $Z_2(\omega)$ given by
$$
Z_j(\omega)=\{\lambda\in \mathbb C\mid (\omega,\lambda)\in Z_j\}
$$
do not intersect. The behavior of the contour $\gamma$ at infinity is given by the following
\begin{defi}\label{d:admissible}
Let $\psi\in (0,\pi)$ be a fixed angle, and let $\omega$ be a regular
point.  A smooth simple contour $\gamma$ on $\mathbb C$ is called
\textbf{admissible} (at $\omega$) if:
\begin{itemize}
\item outside of some compact subset of $\mathbb C$, $\gamma$ is given by the rays
$\arg\lambda=\pm\psi$, and
\item $\gamma$ separates $\mathbb C$ into two regions, $\Gamma_1$ and $\Gamma_2$,
such that sufficiently large positive real numbers lie in $\Gamma_2$,
and $Z_j(\omega)\subset\Gamma_j$ for $j=1,2$. 
\end{itemize}
(Henceforth, we assume that $\arg\lambda\in [-\pi,\pi]$. The contour $\gamma$ and the regions
$\Gamma_j$ are allowed to have several connected components.)
\end{defi}

Existence of admissible contours and convergence of the integral is
guaranteed by the following condition:

\begin{enumerate}
\item[(B)] For any compact $K_\omega\subset \Omega$, there exist constants
$C$ and $R$ such that for $\omega\in K_\omega$ and $|\lambda|\geq R$,
\begin{itemize}
\item for $|\arg\lambda|\leq\psi$, we have
$(\omega,\lambda)\not\in Z_1$ and $\|R_1(\omega,\lambda)\|\leq C/|\lambda|$, and
\item for $|\arg\lambda|\geq\psi$, we have
$(\omega,\lambda)\not\in Z_2$ and $\|R_2(\omega,\lambda)\|\leq C/|\lambda|$.
\end{itemize}
\end{enumerate}

\begin{figure}
\includegraphics{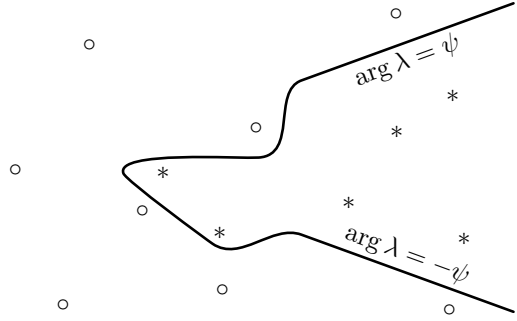}
\caption{An admissible contour. The poles of $R_1$ are denoted by circles and
the poles of $R_2$ are denoted by asterisks.}
\end{figure}

It follows from (B) that there exist admissible contours at every
regular point.  Take a regular point $\omega$, an admissible contour
$\gamma$ at $\omega$, and define
\begin{equation}\label{e:r-omega-int}
R(\omega)={1\over 2\pi i}\int_{\gamma} R_1(\omega,\lambda) \otimes R_2(\omega,\lambda)\,d\lambda.
\end{equation}
Here the orientation of $\gamma$ is chosen so that $\Gamma_1$ always
stays on the left.  The integral above converges and is independent of
the choice of an admissible contour $\gamma$.  Moreover, the set of
regular points is open and $R$ is holomorphic on this set.  (We may
represent $R(\omega)$ as a locally uniform limit of the integral over
the intersection of $\gamma$ with a ball whose radius goes to
infinity.)

The main result of this section is
\begin{prop}\label{l:sep-mer} Assume that $\mathcal H_1$ and $\mathcal H_2$ are two Hilbert
spaces, and $\mathcal H=\mathcal H_1 \otimes \mathcal H_2$ is their Hilbert tensor product.
Let $R_1(\omega,\lambda)$ and $R_2(\omega,\lambda)$ be two families of bounded operators
on $\mathcal H_1$ and $\mathcal H_2$, respectively, for $\omega\in\Omega\subset \mathbb C$
and $\lambda\in \mathbb C$. Assume that $R_1$ and $R_2$ satisfy assumptions
(A)--(B) and the nondegeneracy assumption
\begin{enumerate}
\item[(C)] The set $\Omega_R$ of all regular points is nonempty.
\end{enumerate}
Then the set of all non-regular points is discrete
and the operator $R(\omega)$ defined by \eqref{e:r-omega-int}
is meromorphic in $\omega\in\Omega$ with poles of finite rank. 
\end{prop}

The rest of this section contains the proof of
Proposition~\ref{l:sep-mer}. First, let us establish a normal form for
meromorphic decompositions of families in two variables:

\begin{prop}
Let $T(\omega,\lambda)$ be meromorphic (with values in some Banach space) and
assume that $(\omega_0,\lambda_0)$ lies in the divisor of $T$. Then
we can write near $(\omega_0,\lambda_0)$
$$
T(\omega,\lambda)={S(\omega,\lambda)\over Q(\omega,\lambda)},
$$
where $S$ is holomorphic and $Q$ is a monic polynomial in $\lambda$ of degree $N$
 and  coefficients holomorphic in $\omega$; moreover, $Q(\omega_0,\lambda)=(\lambda-\lambda_0)^N$.
The divisor of $T$ coincides with the set of zeroes of $Q$ near $(\omega_0,\lambda_0)$.
\end{prop}
\begin{proof}
Follows from Definition~\ref{d:mer2} and Weierstrass Preparation Theorem.
\end{proof}

\begin{prop}\label{l:coprime}
Assume that $Q_j(\omega,\lambda)$, $j=1,2$, are two monic polynomials
in $\lambda$ of degrees $N_j$ with coefficients holomorphic in
$\omega$ near $\omega_0$.  Assume also that for some $\omega$, $Q_1$
and $Q_2$ are coprime as polynomials.  Then there exist unique
polynomials $p_1$ and $p_2$ of degree no more than $N_2-1$ and
$N_1-1$, respectively, with coefficients meromorphic in $\omega$ and
such that
$$
1=p_1Q_1+p_2Q_2
$$
when $p_1$ and $p_2$ are well-defined.
\end{prop}
\begin{proof}
The $N_1+N_2$ coefficients of $p_1$ and $p_2$ solve a system of
$N_1+N_2$ linear equations with fixed right-hand side and the matrix
$A(\omega)$ depending holomorphically on $\omega$.  If $\omega$ is
chosen so that $Q_1$ and $Q_2$ are coprime, then the system has a
unique solution; therefore, the determinant of $A(\omega)$ is not
identically zero. The proposition then follows from Cramer's Rule.
\end{proof}

We are now ready to prove that $R(\omega)$ is meromorphic.  It
suffices to show that for each $\omega_0\not\in\Omega_R$ lying in the
closure $\overline{\Omega_R}$, $\omega_0$ is an isolated non-regular
point and $R(\omega)$ has a meromorphic decomposition at $\omega_0$
with finite-dimensional principal part. Indeed, in this case
$\overline{\Omega_R}$ is open; since it is closed and nonempty by (C),
we have $\overline{\Omega_R}=\Omega$ and the statement above applies
to each $\omega_0$.

Let $Z_1(\omega_0)\cap Z_2(\omega_0)=\{\lambda_1,\dots,\lambda_m\}$.
We choose a ball $\Omega_0$ centered at $\omega_0$ and disjoint balls
$U_l$ centered at $\lambda_l$ such that:
\begin{itemize}
\item for $\omega\in\Omega_0$, the set $Z_1(\omega)\cap Z_2(\omega)$
is covered by balls $U_l$ and the set $Z_1(\omega)\cup Z_2(\omega)$
does not intersect the circles $\partial U_l$;
\item for $\omega\in\Omega_0$ and $\lambda\in U_l$, we have $R_j=S_{jl}/Q_{jl}$,
where $S_{jl}$ are holomorphic and $Q_{jl}$ are monic polynomials in $\lambda$
of degree $N_{jl}$ with coefficients
holomorphic in $\omega$, and $Q_{jl}(\omega_0,\lambda)=(\lambda-\lambda_l)^{N_{jl}}$;
\item for $\omega\in\Omega_0$, the set of all roots of  $Q_{jl}(\omega,\cdot)$  coincides with
$Z_j(\omega)\cap U_l$;
\item there exists a contour $\gamma_0$ that does not intersect any $U_l$ and is admissible
for any $\omega\in\Omega_0$
with respect to the sets $Z_j(\omega)\setminus\cup U_l$ in place of $Z_j(\omega)$; moreover,
each $\partial U_l$ lies in the region $\Gamma_1$ with respect to $\gamma_0$ (see Definition~\ref{d:admissible}). 
\end{itemize}

Let us assume that $\omega\in\Omega_0$ is regular. (Such points exist
since $\omega_0$ lies in the closure of $\Omega_R$.) For every $l$,
the polynomials $Q_{1l}(\omega,\lambda)$ and $Q_{2l}(\omega,\lambda)$
are coprime; we find by Proposition~\ref{l:coprime} unique polynomials
$p_{1l}(\omega,\lambda)$ and $p_{2l}(\omega,\lambda)$ such that
$$
1=p_{1l}Q_{1l}+p_{2l}Q_{2l}
$$
and $\deg p_{1l}<N_{2l}$, $\deg p_{2l}<N_{1l}$. The converse is also
true: if all coefficients of $p_{1l}$ and $p_{2l}$ are holomorphic at
some point $\omega$ for all $l$, then $\omega$ is a regular point. It
follows immediately that $\omega_0$ is an isolated non-regular point.

To obtain the meromorphic expansion of $R(\omega)$ near $\omega_0$,
let us take a regular point $\omega\in \Omega_0$ and an admissible
contour $\gamma=\gamma_0+\cdots+\gamma_m$, where $\gamma_0$ is the
$\omega$-independent contour defined above and each $\gamma_l$ is a
contour lying in $U_l$. The integral over $\gamma_0$ is holomorphic
near $\omega_0$, while
$$
\begin{gathered}
\int_{\gamma_l} R_1(\omega,\lambda) \otimes R_2(\omega,\lambda)\,d\lambda\\
=\int_{\gamma_l} S_{1l}(\omega,\lambda) \otimes S_{2l}(\omega,\lambda)
\left(
{p_{1l}(\omega,\lambda)\over Q_{2l}(\omega,\lambda)}
+{p_{2l}(\omega,\lambda)\over Q_{1l}(\omega,\lambda)}
\right)\,d\lambda\\
=\int_{\partial U_l}
p_{1l}S_{1l} \otimes R_2\,d\lambda
=\sum_{j=0}^{N_{2l}-1} p_{1lj}(\omega)\int_{\partial U_l}
(\lambda-\lambda_l)^j S_{1l} \otimes R_2\,d\lambda. 
\end{gathered} 
$$
Here $p_{1lj}(\omega)$ are the coefficients of $p_{1l}$ as a
polynomial of $\lambda-\lambda_l$; they are meromorphic in $\omega$
and the rest is holomorphic in $\omega\in\Omega_0$.

It remains to prove that $R$ has poles of finite rank. It suffices to
show that every derivative in $\omega$ of the last integral above at
$\omega=\omega_0$ has finite rank. Each of these, in turn, is a finite
linear combination of
$$
\int_{\partial U_l}(\lambda-\lambda_l)^j \partial^a_\omega S_{1l}(\omega_0,\lambda)\otimes \partial^b_\omega
R_2(\omega_0,\lambda)\,d\lambda.
$$
However, since $\partial^a_\omega S_{1l}(\omega_0,\lambda)$ is
holomorphic in $\lambda\in U_l$, only the principal part of the
Laurent decomposition of $\partial^b_\omega R_2(\omega_0,\lambda)$ at
$\lambda=\lambda_l$ will contribute to this integral; therefore, the
image of each operator in the principal part of Laurent decomposition
of $R(\omega)$ at $\omega_0$ lies in $\mathcal H_1 \otimes V_2$, where
$V_2$ is a certain finite-dimensional subspace of $\mathcal H_2$. It
remains to show that each of these images also lies in $V_1 \otimes
\mathcal H_2$, where $V_1$ is a certain finite-dimensional subspace of
$\mathcal H_1$.  This is done by the same argument, using the fact
that
$$
\int_{\partial U_l-\gamma_l}R_1(\omega,\lambda) \otimes R_2(\omega,\lambda)\,d\lambda
$$
can be written in terms of $p_{2l}$ and $R_1 \otimes S_{2l}$ and the
integral over $\partial U_l$ is holomorphic at $\omega_0$. The proof
of Proposition~\ref{l:sep-mer} is finished.

\section{Construction of $R_g(\omega)$}

As we saw in the previous section, one can deduce the existence of an
inverse to $P_g=P_r+P_\theta$ and its properties from certain
properties of the inverses to $P_r+\lambda$ and $P_\theta-\lambda$ for
$\lambda\in \mathbb C$. We start with the latter. For $a=0$,
$P_\theta$ is the (negative) Laplace--Beltrami operator for the round
metric on $\mathbb S^2$; therefore, its eigenvalues are given by
$\lambda=l(l+1)$ for $l\in \mathbb Z$, $l\geq 0$.  Moreover, if $\mathcal D'_k$
is the space defined in~\eqref{e:d-prime-k} and there
is an eigenfunction of $P_\theta|_{\mathcal D'_k}$ with eigenvalue
$l(l+1)$, then $l\geq k$. These observations can be generalized to our
case:

\begin{prop}\label{l:angular}
There exists a two-sided inverse
$$
R_\theta(\omega,\lambda)=(P_\theta(\omega)-\lambda)^{-1}
:L^2(\mathbb S^2)\to H^2(\mathbb S^2),\
(\omega,\lambda)\in \mathbb C^2,
$$
with the following properties:

1. $R_\theta(\omega,\lambda)$ is meromorphic with poles of finite rank
in the sense of Definition~\ref{d:2-mer-fin} and it has the following
meromorphic decomposition at $\omega=\lambda=0$:
\begin{equation}\label{e:r-theta-zero}
R_\theta(\omega,\lambda)={S_{\theta0}(\omega,\lambda)\over \lambda-\lambda_\theta(\omega)}
\end{equation}
where $S_{\theta 0}$ and $\lambda_\theta$ are holomorphic in $a$-independent neighborhoods
of zero and
$$
S_{\theta 0}(0,0)=-{1 \otimes 1\over 4\pi},\
\lambda_\theta(\omega)=O(|\omega|^2).
$$

2. There exists a constant $C_\theta$ such that
\begin{equation}\label{e:angular-est-1}
\|R_\theta(\omega,\lambda)\|_{L^2(\mathbb S^2)\cap \mathcal D'_k
\to L^2(\mathbb S^2)}\leq {C_\theta\over |k|^2}
\text{ for }
|\lambda|\leq k^2/2,\ |k|\geq C_\theta|a\omega|.
\end{equation}
and
\begin{equation}\label{e:angular-est-2}
\begin{gathered}
\|R_\theta(\omega,\lambda)\|_{L^2(\mathbb S^2)\cap \mathcal D'_k
\to L^2(\mathbb S^2)}\leq {2\over |\Imag\lambda|}\\
\text{for }|\Imag\lambda|> C_\theta|a|(|a\omega|+|k|)|\Imag\omega|.
\end{gathered}
\end{equation}

3. For every $\psi>0$, there exists a constant $C_\psi$ such that
\begin{equation}\label{e:angular-est-3}
\|R_\theta(\omega,\lambda)\|_{L^2(\mathbb S^2)\to L^2(\mathbb S^2)}
\leq {C_\psi\over |\lambda|}
\text{ for }
|\arg\lambda|\geq\psi,\ |\lambda|\geq C_\psi|a\omega|^2.
\end{equation}
\end{prop}
\begin{proof}
1. Recall~\eqref{e:p-r-theta} that $P_\theta(\omega)$ is a holomorphic
family of elliptic second order differential operators on the
sphere. Therefore, for each $\lambda$, the operator
$P_\theta(\omega)-\lambda:H^2(\mathbb S^2)\to L^2(\mathbb S^2)$ is
Fredholm (see for example~\cite[Section~7.10]{t}). By
Proposition~\ref{l:aft-2}, $R_g(\omega,\lambda)$ is a meromorphic
family of operators $L^2\to H^2$.

We now obtain a meromorphic decomposition for $R_\theta$ near zero using the framework of
Grushin problems~\cite[Appendix~C]{e-z}. Let $i_1:\mathbb C\to L^2(\mathbb S^2)$
be the operator of multiplicaton by the constant function $1$
and $\pi_1:H^2(\mathbb S^2)\to \mathbb C$ be the operator mapping
every function to its integral over the standard measure on the round sphere.
Consider the operator $A(\omega,\lambda):H^2 \oplus \mathbb C\to L^2\oplus \mathbb C$ given by
$$
A(\omega,\lambda)=\begin{pmatrix} P_\theta(\omega)-\lambda& i_1\\
\pi_1&0
\end{pmatrix}.
$$
The kernel and cokernel of $P_\theta(0)$ are both one-dimensional and
spanned by $1$, since this is the Laplace--Beltrami operator for a
certain Riemannian metric on the sphere.  (Indeed, by ellipticity
these spaces consist of smooth functions; by self-adjointness, the
kernel and cokernel coincide; one can then apply Green's
formula~\cite[(2.4.8)]{t} to an element of the kernel and itself.)
Therefore~\cite[Theorem~C.1]{e-z}, the operator
$B(\omega,\lambda)=A(\omega,\lambda)^{-1}$ is well-defined at $(0,0)$;
then it is well-defined for $(\omega,\lambda)$ in an $a$-independent
neighborhood of zero. We write
$$
B(\omega,\lambda)=\begin{pmatrix}B_{11}(\omega,\lambda)&B_{12}(\omega,\lambda)\\
B_{21}(\omega,\lambda)&B_{22}(\omega,\lambda)\end{pmatrix}.
$$
Now, by Schur's complement formula we have near $(0,0)$,
$$
R_\theta(\omega,\lambda)=B_{11}(\omega,\lambda)-B_{12}(\omega,\lambda)B_{22}(\omega,\lambda)^{-1}B_{21}(\omega,\lambda).
$$
However, $B_{22}(\omega,\lambda)$ is a holomorphic function of two variables, and
we can find
$$
B_{22}(\omega,\lambda)={\lambda\over 4\pi}+O(|\omega|^2+|\lambda|^2).
$$
(The $\omega$-derivative vanishes at zero since $\partial_\omega P_\omega(0)|_{\mathcal D'_0}=0$.
To compute the $\lambda$-derivative, we use that $B_{12}(0,0)=i_1/4\pi$
and $B_{21}(0,0)=\pi_1/4\pi$.) The decomposition~\eqref{e:r-theta-zero}
now follows by Weierstrass Preparation Theorem.
\smallskip

2. We have $P_\theta(\omega)=P_\theta(0)+P'_\theta(\omega)$, where
$$
P'_\theta(\omega)={(1+\alpha)^2a\omega\over\Delta_\theta}
(-2D_\varphi+a\omega\sin^2\theta)
$$
is a first order differential operator and
$$
P_\theta(0)={1\over\sin\theta}D_\theta(\Delta_\theta\sin\theta D_\theta)
+{(1+\alpha)^2\over\Delta_\theta\sin^2\theta}D_\varphi^2:H^2(\mathbb S^2)\to L^2(\mathbb S^2)
$$
satisfies $P_\theta(0)\geq k^2$ on $\mathcal D'_k$; therefore,
if $u\in H^2(\mathbb S^2)\cap \mathcal D'_k$, then
$$
\|u\|_{L^2}\leq {\|(P_\theta(0)-\lambda)u\|_{L^2}\over d(\lambda,k^2+\mathbb R^+)}.
$$ 
Since
$$
\|P'_\theta(\omega)\|_{L^2(\mathbb S^2)\cap \mathcal D'_k\to L^2(\mathbb S^2)}
\leq 2(1+\alpha)^2|a\omega|(|a\omega|+|k|),
$$
we get
\begin{equation}\label{e:angular-internal}
\|u\|_{L^2}\leq {\|(P_\theta(\omega)-\lambda)u\|_{L^2}\over d(\lambda,k^2+\mathbb R^+)-C_1|a\omega|(|a\omega|+|k|)},
\end{equation}
provided that the denominator is positive. Here $C_1$ is a global constant.

Now, if $|\lambda|\leq k^2/2$, then 
$d(\lambda,k^2+ \mathbb R^+)\geq k^2/2$ and
$$
d(\lambda,k^2+\mathbb R^+)-C_1|a\omega|(|a\omega|+|k|)\geq {k^2\over 4}
\text{ for }|k|\geq 8(1+C_1)|a\omega|;
$$
together with~\eqref{e:angular-internal}, this proves~\eqref{e:angular-est-1}.

To prove~\eqref{e:angular-est-2}, introduce
$$
\Imag P_\theta(\omega)={1\over 2}(P_\theta(\omega)-P_\theta(\omega)^*)={2(1+\alpha)^2\over\Delta_\theta}
a\Imag\omega(a\Real\omega\sin^2\theta-D_\varphi);
$$
we have
$$
\|\Imag P_\theta(\omega)\|_{L^2(\mathbb S^2)\cap \mathcal D'_k\to L^2(\mathbb S^2)}\leq
2(1+\alpha)^2|a\Imag\omega|(|a\omega|+|k|).
$$
However, for $u\in H^2(\mathbb S^2)\cap \mathcal D'_k$,
$$
\begin{gathered}
\|(P_\theta(\omega)-\lambda)u\|\cdot \|u\|
\geq |\Imag ((P_\theta(\omega)-\lambda)u,u)|
\geq |\Imag\lambda|\cdot \|u\|^2-|(\Imag P_\theta(\omega) u,u)|\\
\geq (|\Imag\lambda|-2(1+\alpha)^2|a|(|a\omega|+|k|)|\Imag\omega|)\|u\|^2
\end{gathered}
$$
and we are done if $C_\theta\geq 4(1+\alpha)^2$.
\smallskip

3. If $|\arg\lambda|\geq\psi$, then
$d(\lambda,k^2+\mathbb R^+)\geq (k^2+|\lambda|)/C_2$;
here $C_2$ is a constant depending on $\psi$. We have then
$$
d(\lambda,k^2+\mathbb R^+)-C_1|a\omega|(|a\omega|+|k|)\geq {1\over C_2}|\lambda|-C_3|a\omega|^2
$$
for some constant $C_3$, and we are done by~\eqref{e:angular-internal}.
\end{proof}

The analysis of the radial operator $P_r$ is more complicated.
In Sections~4--6, we prove
\begin{prop}\label{l:radial}
There exists a family of operators
$$
R_r(\omega,\lambda,k):L^2_{\comp}(r_-,r_+)\to H^2_{\loc}(r_-,r_+),\
(\omega,\lambda)\in \mathbb C^2,
$$
with the following properties:

1. For each $k\in \mathbb Z$, $R_r(\omega,\lambda,k)$ is meromorphic
with poles of finite rank in the sense of
Definition~\ref{d:2-mer-fin}, and
$(P_r(\omega,k)+\lambda)R_r(\omega,\lambda,k)f=f$ for each $f\in
L^2_{\comp}(r_-,r_+)$. Also, for $k=0$, $R_r$ admits the following meromorphic
decomposition near $\omega=\lambda=0$:
\begin{equation}\label{e:r-r-zero}
R_r(\omega,\lambda,0)={S_{r0}(\omega,\lambda)\over\lambda-\lambda_r(\omega)},
\end{equation}
where $S_{r0}$ and $\lambda_r$ are holomorphic in $a$-independent neighborhoods of zero and
$$
\begin{gathered}
S_{r0}(0,0)={1 \otimes 1\over r_+-r_-},\\
\lambda_r(\omega)={i(1+\alpha)(r_+^2+r_-^2+2a^2)\over r_+-r_-}\omega+O(|\omega|^2).
\end{gathered}
$$

2. Take $\delta_r>0$. Then there exist $\psi>0$ and $C_r$ such that
for 
\begin{equation}\label{e:radial-cond}
|\lambda|\geq C_r,\
|\arg\lambda|\leq \psi,\
|ak|^2\leq |\lambda|/C_r,\
|\omega|^2\leq |\lambda|/C_r,
\end{equation}
$(\omega,\lambda,k)$ is not a pole of $R_r$ and we have
\begin{equation}\label{e:radial-est}
\|1_{K_r}R_r(\omega,\lambda,k)1_{K_r}\|_{L^2\to L^2}\leq {C_r\over |\lambda|}.
\end{equation}
Also, there exists $\delta_{r0}>0$ such that, if
$K_+=[r_+-\delta_{r0},r_+]$ and $K_-=[r_-,r_-+\delta_{r0}]$,
then for each $N$ there exists a constant $C_N$ such that
under the conditions~\eqref{e:radial-cond}, we have
\begin{equation}\label{e:radial-est-1.5}
\|1_{K_\pm}|r-r_\pm|^{iA_\pm^{-1}(1+\alpha)((r_\pm^2+a^2)\omega-ak)}
R_r(\omega,\lambda,k)1_{K_r}\|_{L^2\to C^N(K_\pm)}\leq {C_N\over |\lambda|^N}.
\end{equation}

3. There exists a constant $C_\omega$ such that
$R_r(\omega,\lambda,k)$ does not have any poles for real $\lambda$ and
real $\omega$ with $|\omega|> C_\omega|ak|$.

4. Assume that $R_r$ has a pole at $(\omega,\lambda,k)$. Then there
exists a nonzero solution $u\in C^\infty(r_-,r_+)$ to the equation
$(P_r(\omega,k)+\lambda)u=0$ such that the functions
$$
|r-r_\pm|^{iA_\pm^{-1}(1+\alpha)((r_\pm^2+a^2)\omega-ak)}u(r)
$$
are real analytic at $r_\pm$, respectively.

5. Take $\delta_r>0$. Then there exists $C_{1r}>0$ such that for
\begin{equation}\label{e:radial-cond-2}
\Imag\omega>0,\
|ak|\leq  |\omega|/C_{1r},\
|\Imag\lambda|\leq |\omega|\cdot\Imag\omega/C_{1r},\
\Real\lambda\geq -|\omega|^2/C_{1r},
\end{equation}
$(\omega,\lambda,k)$ is not a pole of $R_r$ and we have
\begin{equation}\label{e:radial-est-2}
\|1_{K_r}R_r(\omega,\lambda,k)1_{K_r}\|_{L^2\to L^2}\leq {C_{1r}\over |\omega|\Imag\omega}.
\end{equation}
\end{prop}

Given these two propositions, we can now prove Theorems~\ref{t:r-g-omega-k}--\ref{t:upper-half-plane}:

\begin{proof}[Proof of Theorem~\ref{t:r-g-omega-k}]
Take $k\in \mathbb Z$ and an arbitrary $\delta_r>0$; put
$\mathcal H_1=L^2(K_r)$, $\mathcal H_2=L^2(\mathbb S^2)\cap \mathcal
D'_k$, $R_1(\omega,\lambda)=R_r(\omega,\lambda,k)$, and
$R_2(\omega,\lambda)=R_\theta(\omega,\lambda)|_{\mathcal D'_k}$;
finally, let the angle $\psi$ of admissible contours at infinity be
chosen as in Proposition~\ref{l:radial}.  We now apply
Proposition~\ref{l:sep-mer}. Condition (A) follows from the first parts of
Propositions~\ref{l:angular} and~\ref{l:radial}. Condition (B) follows
from~\eqref{e:angular-est-3} and part~2 of
Proposition~\ref{l:radial}. Finally, condition (C) holds because every
$\omega\in \mathbb R$ with $|\omega|>C_\omega|ak|$, where
$C_\omega$ is the constant from part~3 of Proposition~\ref{l:radial}, is regular.
Indeed, $P_\theta(\omega)$ is self-adjoint and thus has only real eigenvalues. Now,
by Proposition~\ref{l:sep-mer} we can use~\eqref{e:r-omega-int} to define
$R_g(\omega,k)$ as a meromorphic family of operators on $L^2(M_K)\cap \mathcal D'_k$ with
poles of finite rank. This can be done for any $\delta_r>0$;
therefore, $R_g(\omega,k)$ is defined as an operator
$L^2_{\comp}(M)\cap \mathcal D'_k\to
L^2_{\loc}(M)\cap \mathcal D'_k$.

Let us now prove that $P_g(\omega,k)R_g(\omega,k)f=f$ in the sense of
distributions for each $f\in L^2_{\comp}$. We will use the method
of Proposition~\ref{l:finite-sep-mer}. Assume that $\omega$ is a
regular point, so that $R_g(\omega,k)$ is well-defined. By
analyticity, we can further assume that $\omega$ is real, so that
$L^2(\mathbb S^2)\cap \mathcal D'_k$ has an orthonormal basis of
eigenfunctions of $P_\theta(\omega)$. Then it suffices to prove that
$$
I=(R_g(\omega,k)(f_r(r) f_\theta(\theta,\varphi)),P_g(\omega)(h_r(r)h_\theta(\theta,\varphi)))
=(f_r,h_r)\cdot (f_\theta,h_\theta),
$$
where $f_r,h_r\in C_0^\infty(r_-,r_+)$, $h_\theta\in C^\infty(\mathbb
S^2)\cap \mathcal D'_k$, and $f_\theta\in \mathcal D'_k$ satisfies
$$
P_\theta(\omega)f_\theta=\lambda_0 f_\theta,\
\lambda_0\in \mathbb R.
$$
Take an admissible contour $\gamma$; then
$$
\begin{gathered}
I
={1\over 2\pi i}\int_\gamma (R_r(\omega,\lambda,k)f_r,P_r(\omega,k)h_r)\cdot(R_\theta(\omega,\lambda)f_\theta,h_\theta)\\
+(R_r(\omega,\lambda,k)f_r,h_r)\cdot (R_\theta(\omega,\lambda)f_\theta,P_\theta(\omega)h_\theta)\,d\lambda.
\end{gathered}
$$
However,
$$
R_\theta(\omega,\lambda)f_\theta={f_\theta\over \lambda_0-\lambda}.
$$
It then follows from condition (B) that we can replace $\gamma$ by a
closed bounded contour $\gamma'$ which contains $\lambda_0$, but no
poles of $R_r$. (To obtain $\gamma'$, we can cut off the infinite ends
of $\gamma$ sufficiently far and connect the resulting two endpoints
by the arc $-\psi\leq\arg\lambda\leq\psi$; the integral over the arc
can be made arbitrarily small.)  Then
$$
\begin{gathered} 
I={1\over 2\pi i}\int_{\gamma'} ((1-\lambda R_r(\omega,\lambda,k))f_r,h_r)\cdot (R_\theta(\omega,\lambda)f_\theta,h_\theta)\\
+(R_r(\omega,\lambda,k)f_r,h_r)\cdot ((1+\lambda R_\theta(\omega,\lambda))f_\theta,h_\theta)\,d\lambda\\
={1\over 2\pi i}\int_{\gamma'} (f_r,h_r)\cdot (R_\theta(\omega,\lambda)f_\theta,h_\theta)
+(R_r(\omega,\lambda,k)f_r,h_r)\cdot (f_\theta,h_\theta)\,d\lambda\\
={1\over 2\pi i}\int_{\gamma'} {(f_r,h_r)(f_\theta,h_\theta)\over \lambda_0-\lambda}\,d\lambda
=(f_r,h_r)(f_\theta,h_\theta),
\end{gathered}
$$
which finishes the proof.

Finally, the operator $P_g(\omega,k)$ is the restriction to $\mathcal
D'_k$ of the elliptic differential operator on $M$ obtained from $P_g(\omega)$ by replacing $D_\varphi$ by $k$
in the second term of~\eqref{e:p-g-omega}. Therefore, by elliptic
regularity (see for example~\cite[Section~7.4]{t})
the operator $R_g(\omega,k)$ acts into $H^2_{\loc}$.
\end{proof}

Next, Theorem~\ref{t:r-g-omega} follows from Theorem~\ref{t:r-g-omega-k},
the fact that the operator $P_g(\omega)$ is elliptic on $M_K$ for small $a$
(to get $H^2$ regularity instead of $L^2$),
and the following estimate on $R_g(\omega,k)$ for large values of $k$:

\begin{prop}\label{l:large-k-estimate}
Fix $\delta_r>0$. Then there exists $a_0>0$ and a constant $C_k$
such that for $|a|<a_0$ and $|k|\geq C_k(1+|\omega|)$, $\omega$ is not
a pole of $R_g(\cdot,k)$ and we have
\begin{equation}\label{e:large-k-estimate}
\|1_{M_K}R_g(\omega,k)1_{M_K}\|_{L^2\cap \mathcal D'_k\to L^2}
\leq {C_k\over |k|^2}.
\end{equation}
\end{prop}
\begin{proof}
Let $\psi,C_r$ be the constants from part 2 of Proposition~\ref{l:radial} and $C_\theta,C_\psi$
be the constants from Proposition~\ref{l:angular}.
Put $\lambda_0=k^2/3$; if $C_k$ is large enough, then
$$
|k|> 1+C_\theta |a\omega|,\
\lambda_0>C_\psi |a\omega|^2+C_r(1+|\omega|^2).
$$
Take the contour $\gamma$ consisting
of the rays $\{\arg\lambda=\pm\psi,|\lambda|\geq \lambda_0\}$ and the
arc $\{|\lambda|=\lambda_0,\ |\arg\lambda|\leq\psi\}$. By~\eqref{e:angular-est-1}
and~\eqref{e:angular-est-3}, all poles of $R_\theta$ lie inside
$\gamma$ (namely, in the region $\{|\lambda|\geq\lambda_0,\ |\arg\lambda|\leq\psi\}$),
and
\begin{equation}\label{e:theorem-2-est-theta}
\|R_\theta(\omega,\lambda)\|_{L^2(\mathbb S^2)\cap \mathcal D'_k\to L^2(\mathbb S^2)}
\leq {C\over |\lambda|}
\end{equation}
for each $\lambda$ on $\gamma$. Now, suppose that
$|a|<a_0=(3C_r)^{-1/2}$;
then \eqref{e:radial-cond} is satisfied inside $\gamma$
and \eqref{e:large-k-estimate} follows from~\eqref{e:r-omega-int},
\eqref{e:radial-est}, and~\eqref{e:theorem-2-est-theta}.
\end{proof}

\begin{proof}[Proof of Theorem~\ref{t:global-outgoing}] 1.
Fix $\delta_r>0$ such that $\supp f\subset M_K$.
Take an admissible contour~$\gamma$; then by~\eqref{e:r-omega-int}
and the fact that the considered functions are in $\mathcal D'_k$,
\begin{equation}\label{e:theorem-3.5-int}
v_\pm={1\over 2\pi i}\int_\gamma 
(R_r^\pm(\omega,\lambda,k) \otimes R_\theta(\omega,\lambda))f\,d\lambda,
\end{equation}
where
$$
R_r^\pm(\omega,\lambda,k)=|r-r_\pm|^{iA_\pm^{-1}(1+\alpha)((r_\pm^2+a^2)\omega-ak)}R_r(\omega,\lambda,k).
$$
By part~2 of Proposition~\ref{l:radial}, we may choose
compact sets $K_\pm$ containing $r_\pm$ such that for
each $N$, there exists a constant $C_N$ (depending on $\omega$, $k$, and~$\gamma$) such that
$$
\|1_{K_\pm}R_r^\pm(\omega,\lambda,k)1_{K_r}\|_{L^2\to C^N(K_\pm)}\leq {C_N\over 1+|\lambda|},\ \lambda\in\gamma. 
$$
(The estimate is true over a compact portion of $\gamma$ since the image of $R_r^\pm$
consists of functions smooth at $r=r_\pm$, by the construction in Section~4.)
Now, by~\eqref{e:angular-est-3} we get for some constant $C'_N$,
$$
\|R_r^\pm(\omega,\lambda,k) \otimes R_\theta(\omega,\lambda))f\|_{C^N(K_\pm;L^2(\mathbb S^2))}
\leq {C'_N\|f\|_{L^2}\over 1+|\lambda|^2};
$$
by~\eqref{e:theorem-3.5-int}, $v_\pm\in C^\infty(K_\pm;L^2(\mathbb S^2))$.

Now, since $(P_r+P_\theta)u=f$ and  (assuming that $K_\pm\cap K_r=\emptyset$) $f|_{K_\pm\times \mathbb S^2}=0$,
we have $(P_r^\pm(\omega,k)+P_\theta(\omega))v_\pm=0$ on $K_\pm\times \mathbb S^2$, where
$$
P_r^\pm(\omega,k)=|r-r_\pm|^{iA_\pm^{-1}(1+\alpha)((r_\pm^2+a^2)\omega-ak)}
P_r(\omega,k)|r-r_\pm|^{-iA_\pm^{-1}(1+\alpha)((r_\pm^2+a^2)\omega-ak)}
$$
has smooth coefficients on $K_\pm$ (see Section~4). Then for each $N$,
$$
P_\theta^Nv_\pm=(-P_r^\pm)^Nv_\pm\in C^\infty(K_\pm;L^2(\mathbb S^2));
$$
since $P_\theta$ is elliptic, we get $v_\pm\in C^\infty(K_\pm;H^{2N}(\mathbb S^2))$.
Therefore, $v_\pm\in C^\infty(K_\pm\times \mathbb S^2)$.

\smallskip

2. Let $\omega$ be a pole of $R_g(\omega,k)$.  Then $\omega$ is not a
regular point; therefore, there exists $\lambda\in \mathbb C$ such
that $(\omega,\lambda)$ is a pole of both $R_r$ and $R_\theta$. This
gives us functions $u_r(r)$ and $u_\theta(\theta,\varphi)\in \mathcal
D'_k$ such that $(P_r(\omega,k)+\lambda)u_r=0$ and
$(P_\theta(\omega)-\lambda)u_\theta=0$. It remains to take $u=u_r
\otimes u_\theta$ and use part 4 of Proposition~\ref{l:radial}.
\end{proof}

The following fact will be used in the proof of Theorem~\ref{t:upper-half-plane},
as well as in Section~7:
\begin{prop}\label{l:smart-contour}
Fix $\delta_r>0$. Let $\psi,C_r$ be the constants
from part~2 of Proposition~\ref{l:radial},
$C_\theta,C_\psi$ be the constants from Proposition~\ref{l:angular},
and $C_k$ be the constant from Proposition~\ref{l:large-k-estimate}.
Take $\omega\in \mathbb C$ and put
$$
L=(C_r(1+C_k)^2+C_\psi)(1+|\omega|)^2.
$$
Assume that $a$ is small enough so that Proposition~\ref{l:large-k-estimate} applies and
suppose that $\omega$ and $l_1,l_2>0$ are chosen so that
\begin{equation}\label{e:smart-contour-cond}
l_1\geq C_\psi |a\omega|^2,\
l_2\geq C_\theta|a|(|a\omega|+C_k(1+|\omega|))|\Imag\omega|,\
l_2\leq L\sin\psi.
\end{equation}
Also, assume that for all $\lambda$ and $k$ satisfying
\begin{equation}\label{e:smart-contour-2-cond}
|k|\leq C_k(1+|\omega|),\
-l_1\leq\Real\lambda\leq L,\
|\Imag\lambda|\leq l_2,
\end{equation}
we have the estimate
\begin{equation}\label{e:smart-contour-2-est}
\|1_{K_r}R_r(\omega,\lambda,k)1_{K_r}\|_{L^2\to L^2}\leq C_1
\end{equation}
for some constant $C_1$ independent of $\lambda$ and $k$.
Then $\omega$ is not a resonance and
\begin{equation}\label{e:smart-contour-result}
\|R_g(\omega)\|_{L^2(M_K)\to L^2(M_K)}\leq C_2\bigg({1\over 1+|\omega|^2}
+{1+C_1(l_1+1+|\omega|^2)\over l_2}+{C_1 l_2\over l_1}\bigg)
\end{equation}
for a certain global constant $C_2$.
\end{prop}
\begin{proof} First of all, by Proposition~\ref{l:large-k-estimate},
it suffices to establish the estimate~\eqref{e:smart-contour-result}
for the operator $R_g(\omega,k)$, where $|k|\leq C_k(1+|\omega|)$.
Now, by~\eqref{e:r-omega-int}, it suffices to construct
an admissible contour in the sense of Definition~\ref{d:admissible}
and estimate the norms of $R_r$ and $R_\theta$ on this contour. 
We take the contour $\gamma$ composed of:
\begin{itemize}
\item the rays $\gamma_{1\pm}=\{\arg\lambda=\pm\psi,|\lambda|\geq L\}$;
\item the arcs $\gamma_{2\pm}=\{|\arg\lambda|\leq\psi,\ |\lambda|=L,\ \pm\Imag\lambda\geq l_2\}$;
\item the segments $\gamma_{3\pm}$ of the lines $\{\Imag\lambda=\pm l_2\}$ connecting
$\gamma_{2\pm}$ with $\gamma_4$;
\item the segment $\gamma_4=\{\Real\lambda=-l_1,\ |\Imag\lambda|\leq l_2\}$.
\end{itemize}
Then $\gamma$ divides the complex plane into two domains;
we refer to the domain containing positive real numbers as $\Gamma_2$
and to the other domain as $\Gamma_1$. We claim that $R_\theta(\omega,\cdot)|_{\mathcal D'_k}$ has no poles
in $\Gamma_1$, $R_r(\omega,\cdot,k)$ has no poles in $\Gamma_2$, and the $L^2\to L^2$ operator norm estimates
\begin{align}
\label{e:smart-contour-est-1}
\|R_\theta(\omega,\lambda)\|\leq C/|\lambda|,\
\|1_{K_r}R_r(\omega,\lambda,k)1_{K_r}\|\leq C/|\lambda|,\
\lambda&\in\gamma_{1\pm};\\
\label{e:smart-contour-est-2}
\|R_\theta(\omega,\lambda)|_{\mathcal D'_k}\|\leq C/l_2,\
\|1_{K_r}R_r(\omega,\lambda,k)1_{K_r}\|\leq C/(1+|\omega|^2),\
\lambda&\in\gamma_{2\pm};\\
\label{e:smart-contour-est-3}
\|R_\theta(\omega,\lambda)|_{\mathcal D'_k}\|\leq C/l_2,\
\|1_{K_r}R_r(\omega,\lambda,k)1_{K_r}\|\leq C_1,\
\lambda&\in\gamma_{3\pm};\\
\label{e:smart-contour-est-4}
\|R_\theta(\omega,\lambda)\|\leq C/l_1,\
\|1_{K_r}R_r(\omega,\lambda,k)1_{K_r}\|\leq C_1,\
\lambda&\in\gamma_4
\end{align}
hold for some global constant $C$; then~\eqref{e:smart-contour-result} follows
from these estimates and~\eqref{e:r-omega-int}.

\begin{figure}
\includegraphics{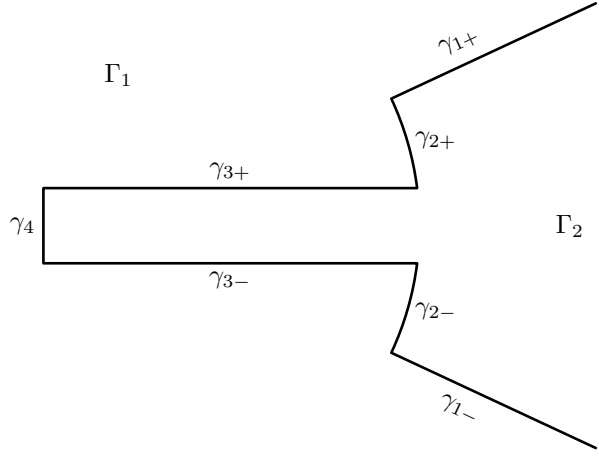}
\caption{The admissible contour $\gamma$ used in Proposition~\ref{l:smart-contour}}
\end{figure}

First, we prove that $R_\theta(\omega,\cdot)|_{\mathcal D'_k}$ has no poles $\lambda\in\Gamma_1$. First of all,
assume that $|\lambda|\geq L$. Then $|\arg\lambda|\geq\psi$ and we can apply
part~3 of Proposition~\ref{l:angular}; we also get the first half of~\eqref{e:smart-contour-est-1}.
Same argument works for $\Real\lambda\leq -l_1$, and we get the first half of~\eqref{e:smart-contour-est-4}.
We may now assume that $|\lambda|\leq L$ and $\Real\lambda\geq -l_1$; it follows
that $|\Imag\lambda|\geq l_2$. But in that case, we can apply~\eqref{e:angular-est-2},
and we get the first halves of~\eqref{e:smart-contour-est-2} and~\eqref{e:smart-contour-est-3}. 

Next, we prove that $R_r(\omega,\cdot,k)$
has no poles $\lambda\in\Gamma_2$. First of all,
assume that $|\lambda|\geq L$ and $\Real\lambda\geq 0$. Then
$|\arg\lambda|\leq\psi$ and we can apply
part~2 of Proposition~\ref{l:radial}; we also get the second halves of~\eqref{e:smart-contour-est-1}
and~\eqref{e:smart-contour-est-2}. Now, in the opposite case, \eqref{e:smart-contour-2-cond}
is satisfied and we can use~\eqref{e:smart-contour-2-est} to get the second
halves of~\eqref{e:smart-contour-est-3} and~\eqref{e:smart-contour-est-4}.
\end{proof}

\begin{proof}[Proof of Theorem~\ref{t:upper-half-plane}]
First, we take care of the resonances near zero.
By Proposition~\ref{l:large-k-estimate}, we can assume that $k$ is bounded
by some constant. Next, if $\omega=0$ and $a=0$, then $R_g(\omega,k)$ only
has a pole for $k=0$, and in the latter case, $\lambda=0$ is the only
common pole of $R_\theta(0,\cdot)$ and $R_r(0,\cdot,0)$. (In fact, the poles
of $R_\theta(0,\cdot)|_{\mathcal D'_k}$ are given by $\lambda=l(l+1)$ for $l\geq |k|$; an integration
by parts argument shows that $R_r(0,\cdot,k)$ cannot have poles with $\Real\lambda>0$.)
The sets of poles of the resolvents $R_\theta(\omega,\lambda)|_{\mathcal D'_k}$
and $R_r(\omega,\lambda,k)$ depend continuously on $a$ in the sense that, if 
there are no poles of one of these resolvents for $(\omega,\lambda)$ in a fixed
compact set for $a=0$, then this is still true for $a$ small enough. It follows 
from here and the first parts of Propositions~\ref{l:angular} and~\ref{l:radial} that
there exists $\varepsilon_\omega,\varepsilon_\lambda>0$ such that for $a$ small enough,
\begin{itemize}
\item $R_g(\omega,k)$ does not have poles in $\{|\omega|\leq\varepsilon_\omega\}$
unless $k=0$;
\item if $|\omega|\leq\varepsilon_\omega$, then all common poles of $R_\theta(\omega,\cdot)|_{\mathcal D'_0}$
and $R_r(\omega,\cdot,0)$ lie in $\{|\lambda|\leq\varepsilon_\lambda\}$;
\item the decompositions~\eqref{e:r-theta-zero} and~\eqref{e:r-r-zero}
hold for $|\omega|\leq\varepsilon_\omega$, $|\lambda|\leq\varepsilon_\lambda$;
\item we have $\lambda_r(\omega)\neq\lambda_\theta(\omega)$ for $0<|\omega|\leq\varepsilon_\omega$. 
\end{itemize}
It follows immediately that $\omega=0$ is the only pole of $R_g$ in $\{|\omega|\leq\varepsilon_\omega\}$.
To get the meromorphic decomposition, we repeat the argument at the end of Section~2 in our
particular case. Note that for small $\omega\neq 0$,
$$
R_g(\omega,0)={1\over 2\pi i}\int_\gamma R_r(\omega,\lambda,0) \otimes R_\theta(\omega,\lambda)|_{\mathcal D'_0}\,d\lambda+\Hol(\omega)
$$
Here $\gamma$ is a small contour surrounding $\lambda_\theta(\omega)$, but not $\lambda_r(\omega)$;
the integration is done in the clockwise direction; $\Hol$ denotes a family of operators
holomorphic near zero.
By~\eqref{e:r-theta-zero} and~\eqref{e:r-r-zero},
we have
$$
\begin{gathered}
R_g(\omega,0)=\Hol(\omega)+{1\over 2\pi i}\int_\gamma {S_{r0}(\omega,\lambda) \otimes S_{\theta 0}(\omega,\lambda)
\over (\lambda-\lambda_r(\omega))(\lambda-\lambda_\theta(\omega))}\,d\lambda\\
=\Hol(\omega)+{1\over \lambda_r(\omega)-\lambda_\theta(\omega)}
{1\over 2\pi i}\int_\gamma (S_{r0}(\omega,\lambda) \otimes S_{\theta 0}(\omega,\lambda))
\bigg({1\over\lambda-\lambda_r(\omega)}-{1\over\lambda-\lambda_\theta(\omega)}\bigg)\,d\lambda\\
=\Hol(\omega)+{1\over \lambda_r(\omega)-\lambda_\theta(\omega)}S_{r0}(\omega,\lambda_\theta(\omega))
\otimes S_{\theta 0}(\omega,\lambda_\theta(\omega))\\
=\Hol(\omega)+{i(1 \otimes 1)\over 4\pi(1+\alpha)(r_+^2+r_-^2+2a^2)\omega}.
\end{gathered} 
$$
 
Now, let us consider the case $|\omega|>\varepsilon_\omega,\ \Imag\omega>0$.
We will apply Proposition~\ref{l:smart-contour}
with $l_1=|\omega|^2/C_{1r},\ l_2=|\omega|\Imag\omega/C_{1r}$.
Here $C_{1r}$ is the constant in Proposition~\ref{l:radial}. Then~\eqref{e:smart-contour-cond}
is true for small $a$
and~\eqref{e:smart-contour-2-est}
follows from~\eqref{e:smart-contour-2-cond} for small $a$
by part~5 of Proposition~\ref{l:radial}, with $C_1=C_{1r}/(|\omega|\Imag\omega)$. It remains
to use~\eqref{e:smart-contour-result}.

Finally, assume that $\omega$ is a real $k$-resonance and $|\omega|>\varepsilon_\omega$. Then by
Proposition~\ref{l:large-k-estimate}, and part~3 of Proposition~\ref{l:radial},
if $a$ is small enough, then the operator $R_r(\omega,\cdot,k)$ cannot have
a pole for $\lambda\in \mathbb R$. However, the operator $P_\theta(\omega)$ is
self-adjoint and thus only has real eigenvalues, a contradiction.  
\end{proof}

\section{Construction of the radial resolvent}

In this section, we prove Proposition~\ref{l:radial}, except for part 2, which is
proved in Section~6.
We start with a change of variables that maps 
$(r_-,r_+)$ to $(-\infty,\infty)$:
\begin{prop}\label{l:regge-wheeler}
Define $x=x(r)$ by 
\begin{equation}\label{e:regge-wheeler}
x=\int_{r_0}^r {ds\over\Delta_r(s)}.
\end{equation}
(Here $r_0\in (r_-,r_+)$ is a fixed number.) Then there exists a
constant $X_0$ such that for $\pm x>X_0$, we have $r=r_\pm\mp F_\pm
(e^{\mp A_\pm x})$, where $F_\pm(w)$ are real analytic on $[0,e^{-A_\pm X_0})$
and holomorphic in the discs $\{|w|< e^{-A_\pm X_0} \}\subset \mathbb C$.
\end{prop}
\begin{proof} We concentrate on the behavior of $x$ near $r_+$.
It is easy to see that $-A_+x(r)=\ln(r_+-r)+G(r)$,
where $G$ is holomorphic near $r=r_+$. Exponentiating, we get
$$
w=e^{-A_+x}=(r_+-r)e^{G(r)}.
$$
It remains to apply the inverse function theorem to solve for $r$ as a function of $w$ near zero.
\end{proof}

After the change of variables $r\to x$, 
we get $P_r(\omega,k)+\lambda=\Delta_r^{-1}P_x(\omega,\lambda,k)$, where
\begin{equation}\label{e:p-x}
\begin{gathered}
P_x(\omega,\lambda,k)=D_x^2+V_x(x;\omega,\lambda,k),
\\
V_x=\lambda\Delta_r-(1+\alpha)^2((r^2+a^2)\omega-ak)^2.
\end{gathered}
\end{equation}
(We treat $r$ and $\Delta_r$ as functions of $x$ now.) We put
\begin{equation}\label{e:omega-pm}
\omega_\pm=(1+\alpha)((r_\pm^2+a^2)\omega-ak),
\end{equation}
so that $V_x(\pm\infty)=-\omega_\pm^2$. Also, by Proposition~4.1,
we get
\begin{equation}\label{e:v-x-complex}
V_x(x)=V_\pm(e^{\mp A_\pm x}),\ \pm x>X_0,
\end{equation}
where $V_\pm(w)$ are functions holomorphic in the discs $\{|w|<e^{-A_\pm X_0}\}$.

We now define outgoing functions:
\begin{defi}\label{d:outgoing} Fix $\omega,k,\lambda$.
A function $u(x)$ (and the corresponding function of $r$) is called outgoing
at $\pm\infty$ iff
\begin{equation}\label{e:u-pm}
u(x)=e^{\pm i\omega_\pm x} v_\pm(e^{\mp A_\pm x}),
\end{equation}
where $v_\pm(w)$ are holomorphic in a neighborhood of zero.
We call $u(x)$ outgoing if it is outgoing at both infinities.
\end{defi}

Let us construct certain solutions outgoing at one of the infinities:
\begin{prop}\label{l:outgoing}
There exist solutions $u_\pm(x;\omega,\lambda,k)$ to the equation $P_x u_\pm=0$
of the form
$$
u_\pm(x;\omega,\lambda,k)=e^{\pm i\omega_\pm x} v_\pm(e^{\mp A_\pm x};\omega,\lambda,k),
$$
where $v_\pm(w;\omega,\lambda,k)$ is holomorphic in $\{|w|<W_\pm\}$ and
\begin{equation}\label{e:v-pm-equ}
v_\pm(0;\omega,\lambda,k)={1\over \Gamma(1-2i\omega_\pm A_\pm^{-1})}.
\end{equation}
These solutions are holomorphic in $(\omega,\lambda)$ and are unique
unless $\nu=2i\omega_\pm A_\pm^{-1}$ is a positive integer.
\end{prop}
\begin{proof} We only construct the function $u_+$.
Let us write the Taylor series for $v_+$ at zero:
$$
v_+(w)=\sum_{j\geq 0} v_j w^j. 
$$
Put $w=e^{-A_+x}$; then the equation $P_xu_+=0$ is equivalent to
$$
((A_+ wD_w-\omega_+)^2+V_x)v_\pm=0.
$$
By~\eqref{e:p-x} and Proposition~\ref{l:regge-wheeler}, $V_x$ is a holomorphic
function of $w$ for $|w|<W_+$.
If $V_x=\sum_{j\geq 0}V_j w^j$ is the corresponding Taylor series,
then we get the following system of
linear equations on the coefficients $v_j$:
\begin{equation}\label{e:outgoing-system}
jA_+(2i\omega_+-jA_+)v_j+\sum_{0<l\leq j}V_l v_{j-l}=0,\ j>0.
\end{equation}
If $\nu$ is not a positive integer, then this system has a unique
solution under the condition $v_0=\Gamma(1-\nu)^{-1}$.  This solution
can be uniquely holomorphically continued to include the cases when
$\nu$ is a positive integer.  Indeed, one defines the coefficients
$v_0,\dots,v_\nu$ by Cramer's Rule using the first $\nu$ equations
in~\eqref{e:outgoing-system} (this can be done since the zeroes of the
determinant of the corresponding matrix match the poles of the gamma
function), and the rest are uniquely determined by the remaining
equations in the system~\eqref{e:outgoing-system}.

We now prove that the series above converges in the disc
$\{|w|<W_+\}$. We take $\varepsilon>0$; then $|V_j|\leq M
(W_+-\varepsilon)^{-j}$ for some constant $M$.  Then one can use
induction and~\eqref{e:outgoing-system} to see that $|v_j|\leq
C(W_+-\varepsilon)^{-j}$ for some constant $C$. Therefore, the Taylor
series for $v$ converges in the disc $\{|w|<W_+-\varepsilon\}$; since
$\varepsilon$ was arbitrary, we are done.
\end{proof}

The condition~\eqref{e:v-pm-equ} makes it possible for $u_\pm$ to be zero for certain
values of $\omega_\pm$. However, we have the following
\begin{prop}\label{l:exceptional}
Assume that one of the solutions $u_\pm$ is identically zero. Then every solution $u$
to the equation $P_xu=0$ is outgoing at the corresponding infinity.
\end{prop}
\begin{proof} Assume that $u_+(x;\omega_0,\lambda_0,k_0)\equiv 0$. (The argument for $u_-$ is similar.)
Put $\nu=2i\omega_{0+} A_+^{-1}$; 
by~\eqref{e:v-pm-equ}, it has to be a positive integer. Similarly to Proposition~\ref{l:outgoing},
we can construct a nonzero solution $u_1$ to the equation $P_x u_1=0$ with
$$
u_1(x)=e^{- i\omega_{0+} x}\tilde v_1(e^{- A_+ x}) 
$$
and $\tilde v_1$ holomorphic at zero. We can see that
$u_1(x)=e^{i\omega_{0+} x}v_1(e^{- A_+ x})$,
where
$v_1(w)=w^\nu \tilde v_1(w)$
is holomorphic; therefore, $u_1$ is outgoing. Note that $u_1(x)=o(e^{ i\omega_{0+} x})$
as $x\to +\infty$.

Now, since $u_+(x;\omega_0,\lambda_0,k_0)\equiv 0$, we can define
$$
u_2(x)=\lim_{\omega\to\omega_0}\Gamma(1-2i\omega_+A_+^{-1})u_+(x;\omega,\lambda_0,k_0);
$$
it will be an outgoing solution to the equation $P_x u_2=0$ and have
$u_2(x)=e^{i\omega_{0+} x}(1+o(1))$ as $x\to+\infty$. We have
constructed two linearly independent outgoing solutions to the
equation $P_xu=0$; since this equation only has a two-dimensional
space of solutions, every its solution must be outgoing.
\end{proof} 

The next statement follows directly from the definition of an outgoing solution and will
be used in later sections:
\begin{prop}\label{l:u-complex} Fix $\delta_r>0$ and let $K_x$ be the image of the set
$K_r=(r_-+\delta_r,r_+-\delta_r)$ under the change of variables $r\to x$.
Assume that $X_0$ is chosen large enough so that Proposition~\ref{l:regge-wheeler} holds
and $K_x\subset (-X_0,X_0)$. Let
$u(x)\in H^2_{\loc}(\mathbb R)$ be any outgoing function in the sense of
Definition~\ref{d:outgoing} and assume that $f=P_x u$ is supported in $K_x$. Then:

1. $u$ can be extended holomorphically to the two half-planes $\{\pm\Real z>X_0\}$
and satisfies the equation $P_z u=0$ in these half-planes, where
$P_z=D_z^2+V_x(z)$ and $V_x(z)$ is well-defined by~\eqref{e:v-x-complex}.

2. If $\gamma$ is a contour in the complex plane given by $\Imag z=F(\Real z)$, $x_-\leq \Real z\leq x_+$,
and $F(x)=0$ for $|x|\leq X_0$, then we can define the restriction to $\gamma$
of the holomorphic extension of $u$ by
$$
u_\gamma(x)=u(x+iF(x))
$$
and $u_\gamma$ satisfies the equation $P_\gamma u_\gamma=f$, where
$$
P_\gamma=\bigg({1\over 1+iF'(x)}D_x\bigg)^2+V_x(x+iF(x)).
$$

3. Assume that $\gamma$ is as above, with $x_\pm=\pm\infty$, and
$F'(x)=c=\const$ for large $|x|$. Then $u_\gamma(x)=O(e^{\mp\Imag((1+ic)\omega_\pm)x})$
as $x\to\pm\infty$. As a consequence, if $\Imag((1+ic)\omega_\pm)>0$,
then $u_\gamma(x)\in H^2(\mathbb R)$.
\end{prop}

We are now ready to prove Proposition~\ref{l:radial}.

\begin{proof}[Proof of part 1]
Given the functions $u_\pm$,
define the operator $S_x(\omega,\lambda,k)$ on $\mathbb R$ by its Schwartz kernel
$$
\begin{gathered}
S_x(x,x';\omega,\lambda,k)=u_+(x)u_-(x')[x>x']
+u_-(x)u_+(x')[x<x'].
\end{gathered}
$$
The operator $S_x(\omega,\lambda)$ acts $L^2_{\comp}(\mathbb R) \to
H^2_{\loc}(\mathbb R)$ and $P_xS_x=W(\omega,\lambda,k)$, where the Wronskian
$$
W(\omega,\lambda,k)=u_+(x;\omega,\lambda,k)\cdot \partial_x u_-(x;\omega,\lambda,k)-u_-(x;\omega,\lambda,k)\cdot \partial_xu_+(x;\omega,\lambda,k)
$$
is constant in $x$. Moreover, $W(\omega,\lambda,k)=0$ if and only if
$u_+(x;\omega,\lambda,k)$ and $u_-(x;\omega,\lambda,k)$ are linearly
dependent as functions of $x$. Also, the image of $S_x$ consists of
outgoing functions.

Now, we define the radial resolvent $R_r(\omega,\lambda,k)=R_x(\omega,\lambda,k)\Delta_r$,
where
\begin{equation}\label{e:r-x}
R_x(\omega,\lambda,k)={S_x(\omega,\lambda,k)\over W(\omega,\lambda,k)}.
\end{equation}
It is clear that $R_r$ is a meromorphic family of operators
$L^2_{\comp}\to H^2_{\loc}$ and $(P_r+\lambda)R_r$ is the identity
operator. We now prove that $R_x$, and thus $R_r$, has poles of finite
rank. Fix $k$ and take $(\omega_0,\lambda_0)\in
\{W=0\}$; we need to prove that for every $l$, the principal part 
of the Laurent decomposition of $\partial_\omega^l
R_x(\omega_0,\lambda,k)$ at $\lambda=\lambda_0$ consists of
finite-dimensional operators.  We use induction on $l$. One has
$P_x(\omega,\lambda,k)R_x(\omega,\lambda,k)=1$; differentiating this
identity $l$ times in $\omega$, we get
$$
P_x(\omega_0,\lambda,k) \partial^l_\omega R_x(\omega_0,\lambda,k)=\delta_{l0}1+\sum_{m=1}^l c_{ml}\partial^m_\omega P_x(\omega_0,\lambda,k)\partial^{l-m}_\omega R_x
(\omega_0,\lambda,k).
$$
(Here $c_{ml}$ are some constants.)
The right-hand side has poles of finite rank by the induction hypothesis. Now,
consider the Laurent decomposition
$$
\partial^l_\omega R_x(\omega_0,\lambda,k)=Q(\lambda)+\sum_{j=1}^N {R_j\over (\lambda-\lambda_0)^j}.
$$
Here $Q$ is holomorphic at $\lambda_0$. Multiplying by $P_x$, we get
$$
\sum_{j=1}^N {P_x(\omega_0,\lambda,k) R_j\over (\lambda-\lambda_0)^j}\sim\sum_{j=1}^N {L_j\over (\lambda-\lambda_0)^j}
$$
up to operators holomorphic at $\lambda_0$. Here $L_j$ are some finite-dimensional operators.
We then have
$$
\begin{gathered}
P_x(\omega_0,\lambda_0,k)R_N=L_N,\\
P_x(\omega_0,\lambda_0,k)R_{N-1}=L_{N-1}-(\partial_\lambda P_x(\omega_0,\lambda_0,k))R_N,\ \dots
\end{gathered}
$$
Each of the right-hand sides has finite rank and the kernel of $P_x(\omega_0,\lambda_0,k)$
is two-dimensional; therefore, each $R_j$ is finite-dimensional as required.
(We also see immediately that the image of each $R_j$ consists of smooth functions.)

Finally, we establish the decomposition at zero.  As in part 1 of
Proposition~\ref{l:angular}, it suffices to compute $S_x(0,0,0)$ and
the first order terms in the Taylor expansion of $W$ at $(0,0,0)$. We
have $u_\pm(x;0,0,0)=1$ for all $x$; therefore, $S_x(x,x';0,0,0)=1$.
Next, put $u_{\omega\pm}(x)=\partial_\omega u_\pm(x;0,0,0)$ and
$u_{\lambda\pm}(x)=\partial_\lambda u_\pm(x;0,0,0)$.  By
differentiating the equation $P_x u_\pm=0$ in $\omega$ and $\lambda$
and recalling the boundary conditions at $\pm\infty$, we get
$$
\begin{gathered}
\partial_x^2 u_{\lambda\pm}(x)=\Delta_r,\\
u_{\lambda\pm}(x)=v_{\lambda\pm}(e^{\mp A_\pm x}),\ \pm x\gg 0;\\
\partial_x^2 u_{\omega\pm}(x)=0,\\
u_{\omega\pm}(x)=\pm i(1+\alpha)(r_\pm^2+a^2)x+v_{\omega\pm}(e^{\mp A_\pm x}),\ \pm x\gg 0,
\end{gathered}
$$
for some functions $v_{\lambda\pm},v_{\omega\pm}$ real analytic at zero. We then find
$$
\begin{gathered}
\partial_\lambda W(0,0,0)=\partial_x(u_{-\lambda}-u_{+\lambda})=\int_{-\infty}^\infty \Delta_r\,dx
=r_+-r_-,\\
\partial_\omega W(0,0,0)=\partial_x(u_{-\omega}-u_{+\omega})=-i(1+\alpha)(r_+^2+r_-^2+2a^2).\qedhere
\end{gathered} 
$$
\end{proof}

\begin{proof}[Proof of part 3]
Assume that
$\omega$ and $\lambda$ are both real and $R_r$ has a pole at $(\omega,\lambda,k)$.
Let $u(x)$ be the corresponding resonant state; we know that it has the asymptotics
$$
\begin{gathered}
u_\pm(x)=e^{\pm i\omega_\pm x}U_\pm(1+O(e^{\mp A_\pm x})),\
x\to \pm\infty;\\
\partial_xu_\pm(x)=e^{\pm i\omega_\pm x}U_\pm(\pm i\omega_\pm+O(e^{\mp A_\pm x})),\
x\to \pm\infty 
\end{gathered}
$$
for some nonzero constants $U_\pm$.
Since $V_x(x;\omega,\lambda,k)$ is real-valued,
both $u$ and $\bar u$ solve the equation $(D_x^2+V_x(x))u=0$.
Then the Wronskian $W_u(x)=u\cdot \partial_x \bar u-\bar u\cdot \partial_x u$
must be constant; however,
$$
W_u(x)\to 
\mp 2i \omega_\pm |U_\pm|^2
\text{ as }x\to \pm\infty.
$$
Then we must have $\omega_+\omega_-\leq 0$; it follows immediately that
$|\omega|=O(|ak|)$.
\end{proof}
	
\begin{proof}[Proof of part 4]
First, assume that neither of $u_\pm$ is identically zero.
Then the resolvent $R_x$, and thus $R_r$, has a pole iff the functions
$u_\pm$ are linearly dependent, or, in other words, if there exists a
nonzero outgoing solution $u(x)$ to the equation $P_xu=0$. Now, if one
of $u_\pm$, say, $u_+$, is identically zero, then by
Proposition~\ref{l:exceptional}, $u_-$ will be an outgoing solution at
both infinities.
\end{proof}

\begin{proof}[Proof of part 5]
Assume that $u(x)$ is outgoing and $P_x(\omega,\lambda,k)u=f\in L^2(K_x)$.
Since $\Imag\omega>0$, we have $\Imag\omega_\pm>0$ and thus $u\in H^2(\mathbb R)$.

First, assume that $|\arg\omega-\pi/2|<\varepsilon$, where
$\varepsilon>0$ is a constant to be chosen later. Then
$$
\Real V_x(x)=(1+\alpha)^2(r^2+a^2)^2(\Imag\omega)^2+\Real\lambda\cdot\Delta_r-(1+\alpha)^2((r^2+a^2)\Real\omega-ak)^2;
$$
using~\eqref{e:radial-cond-2}, we can choose $\varepsilon$ and $C_{1r}$ so that
$\Real V_x(x)\geq|\omega|^2/C>0$ for all $x\in \mathbb R$. Then
$$
\begin{gathered}
\|u\|_{L^2(\mathbb R)}\cdot \|f\|_{L^2(\mathbb R)}\geq
\Real \int \bar u(x) (D_x^2+V_x(x))u(x)\,dx\\
\geq\int \Real V_x(x)|u|^2\,dx
\geq C^{-1}|\omega|^2 \|u\|_{L^2(\mathbb R)}^2
\end{gathered}
$$
and~\eqref{e:radial-est-2} follows.

Now, assume that $|\arg\omega-\pi/2|\geq\varepsilon$. Then
$$
\Imag V_x(x)=-2(1+\alpha)^2((r^2+a^2)\Real\omega-ak)(r^2+a^2)\Imag\omega+\Imag\lambda\cdot\Delta_r;
$$
it follows from~\eqref{e:radial-cond-2} that we can choose $C_{1r}$ so that the sign
of $\Imag V_x(x)$ is constant in $x$ (positive if $\arg\omega>\pi/2$ and negative otherwise) and,
in fact, $|\Imag V_x(x)|\geq |\omega|\Imag\omega/C>0$ for all $x$. Then (assuming that $\Imag V_x(x)>0$)
$$
\begin{gathered}
\|u\|_{L^2(\mathbb R)}\cdot \|f\|_{L^2(\mathbb R)}\geq
\Imag \int \bar u(x) (D_x^2+V_x(x))u(x)\,dx\\
=\int \Imag V_x(x)|u|^2\,dx
\geq C^{-1}|\omega|\Imag\omega\|u\|_{L^2(\mathbb R)}^2
\end{gathered}
$$
and~\eqref{e:radial-est-2} follows.
\end{proof}

\section{Preliminaries from semiclassical analysis}

In this section, we list certain facts from semiclassical analysis needed in the
further analysis of our radial operator. For a general introduction to semiclassical
analysis, the reader is referred to~\cite{e-z}.

Let $a(x,\xi)$ belong to the symbol class
$$
S^m=\{a(x,\xi)\in C^\infty(\mathbb R^2)\mid \sup_{x,\xi} \langle\xi\rangle^{|\beta|-m}|\partial^\alpha_x
\partial^\beta_\xi a(x,\xi)|\leq C_{\alpha\beta}\text{ for all }\alpha,\beta\}.
$$
Here $m\in \mathbb R$ and
$\langle\xi\rangle=\sqrt{1+|\xi|^2}$.
Following~\cite[Section~8.6]{e-z},
we define the corresponding semiclassical
pseudodifferential operator $a^w(x,hD_x)$ by the formula
$$
a^w(x,hD_x)u(x)={1\over 2\pi h}\int e^{{i\over h}(x-y)\eta}a\bigg({x+y\over 2},\eta\bigg)
u(y)\,dy d\eta.
$$
Here $h>0$ is the semiclassical parameter.
We denote by $\Psi^m$ the class of all semiclassical pseudodifferential operators
with symbols in $S^m$.
Introduce the semiclassical Sobolev spaces $H^l_h\subset \mathcal D'(\mathbb R)$ with the norm
$\|u\|_{H^l_h}=\|\langle hD_x\rangle^l u\|_{L^2}$;
then for $a\in S^m$, we have
$$
\|a^w(x,hD_x)\|_{H^l_h\to H^{l-m}_h}\leq C,
$$
where $C$ is a constant depending on $a$, but not on $h$. Also, if
$a(x,\xi)\in C_0^\infty(\mathbb R^2)$, then
\begin{equation}\label{e:l2-l-infty}
\|a^w(x,hD_x)\|_{L^2(\mathbb R)\to L^\infty(\mathbb R)}\leq Ch^{-1/2},
\end{equation}
where $C$ is a constant depending on $a$, but not on $h$.
(See~\cite[Theorem~7.10]{e-z} for the proof.)

General facts on multiplication of pseudodifferential operators can be found
in~\cite[Section~8.6]{e-z}. We will need the following: for $a\in S^m$ and $b\in S^n$,
\footnote{We write $A(h)=O_X(h^k)$ for some Fr\'echet space $X$, if for each
seminorm $\|\cdot\|_X$ of $X$, there exists a constant $C$ such that $\|A(h)\|_X\leq Ch^k$.
We write $A(h)=O_X(h^\infty)$ if $A(h)=O_X(h^k)$ for all $k$.}
\begin{gather}\label{e:semiclassical-disjoint-support}
\text{if }\supp a\cap\supp b=\emptyset,\text{ then }a^w(x,hD_x)b^w(x,hD_x)=O_{L^2\to H^N_h}(h^\infty)
\text{ for all }N;\\
a^w(x,hD_x)b^w(x,hD_x)=(ab)^w(x,hD_x)+O_{\Psi^{m+n-1}}(h),\\
[a^w(x,hD_x),b^w(x,hD_x)]=-ih\{a,b\}^w(x,hD_x)+O_{\Psi^{m+n-2}}(h^2).
\end{gather}
Here $\{\cdot,\cdot\}$ is the Poisson bracket, defined by
$\{a,b\}=\partial_\xi a\cdot\partial_x b-\partial_\xi b\cdot\partial_x a$.
Also, if $A\in\Psi^m$, then the adjoint operator $A^*$ also lies in $\Psi^m$
and its symbol is the complex conjugate of the symbol of $A$. 

One can study pseudodifferential operators on
manifolds~\cite[Appendix~E]{e-z}, and on particular on the circle
$\mathbb S^1=\mathbb R/2\pi \mathbb Z$. If $a(x,\xi)=a(\xi)$ is a
symbol on $T^* \mathbb S^1$ that is independent of $x$, then
$a^w(hD_x)$ is a Fourier series multiplier modulo $O(h^\infty)$:
for each $N$,
\begin{equation}\label{e:fourier-series-multiplier}
\text{if }u(x)=\sum_{j\in \mathbb Z} u_j e^{ijx},\text{ then }
a(hD_x)u(x)=\sum_{j\in \mathbb Z} a(hj)u_j e^{ijx}+O_{H^N_h}(h^\infty)\|u\|_{L^2}.
\end{equation}

In the next three propositions, we assume that $P(h)\in\Psi^m$ and
$P(h)=p^w(x,hD_x)+O_{\Psi^{m-1}}(h)$, where $p(x,\xi)\in S^m$.
\begin{prop}\label{l:elliptic} (Elliptic estimate)
Suppose that
the function $\chi\in S^0$ is chosen so that
$|p|\geq \langle\xi\rangle^m/C>0$ on $\supp\chi$ for some $h$-independent
constant $C$. Also, assume that either the set $\supp\chi$ or its complement
is precompact.
Then there exists a constant $C_1$ such that for each $u\in H^m_h$,
\begin{equation}\label{e:elliptic-estimate}
\|\chi^w(x,hD_x)u\|_{H^m_h}\leq C_1\|P(h)u\|_{L^2}+O(h^\infty)\|u\|_{L^2}.
\end{equation}
\end{prop}
\begin{proof} 
The proof follows the standard parametrix construction.
We find a sequence of symbols $q_j(x,\xi;h)\in S^{-m-j}$, $j\geq 0$,
such that for
$$
Q_N(h)=\sum_{0\leq j\leq N} h^j q_j^w(x,hD_x),
$$
we get
\begin{equation}\label{e:elliptic-eq}
(Q_N(h)P(h)-1)\chi^w(x,hD_x)=O_{\Psi^{-N-1}}(h^{N+1});
\end{equation}
applying this operator equation to $u$, we prove the proposition.

We can take any $q_0\in\Psi^{-m}$ such that $q_0=p^{-1}$ near $\supp\chi$;
such a symbol exists under our assumptions. The rest of $q_j$
can be constructed by induction using the equation~\eqref{e:elliptic-eq}.
\end{proof}

\begin{prop}\label{l:garding} (G\r arding inequalities)
Suppose that $\chi\in C_0^\infty(\mathbb R^2)$.

1. If $\Real p\geq 0$ near $\supp\chi$, then there exists a constant $C$
such that for every $u\in L^2$,
\begin{equation}\label{e:sharp-garding}
\Real (P(h)\chi^w u,\chi^w u)\geq -Ch\|\chi^w u\|_{L^2}^2-O(h^\infty)\|u\|_{L^2}^2.
\end{equation}

2. If $\Real p\geq 2\varepsilon>0$ near $\supp\chi$ for some constant $\varepsilon>0$,
then for $h$ small enough and every $u\in L^2$,
\begin{equation}\label{e:garding}
\Real (P(h)\chi^w u,\chi^w u)\geq \varepsilon \|\chi^w u\|_{L^2}^2-O(h^\infty)\|u\|_{L^2}^2.
\end{equation}
\end{prop}
\begin{proof}
1. Take $\chi_1\in C_0^\infty(\mathbb R^2;\mathbb R)$ such that
$\chi_1=1$ near $\supp\chi$, but $\Real p\geq 0$ near
$\supp\chi_1$. Then, apply the standard sharp G\r arding
inequality~\cite[Theorem~4.24]{e-z} to the operator $\chi_1^w
P(h)\chi_1^w$ and the function $\chi^w u$, and
use~\eqref{e:semiclassical-disjoint-support}.

2. Apply part 1 of this proposition to the operator $P(h)-2\varepsilon$.
\end{proof}

\begin{prop}\label{l:exponentiation}(Exponentiation of pseudodifferential operators)
Assume that $G\in C_0^\infty(\mathbb R^2)$, $s\in \mathbb R$,
and define the operator $e^{sG^w}:L^2\to L^2$ as
$$
e^{sG^w}=\sum_{j\geq 0} {(sG^w)^j\over j!}.
$$
Assume that $|s|$ is bounded by an $h$-independent constant.
Then:

1. $e^{sG^w}\in\Psi^0$ is a pseudodifferential operator.

2. $e^{sG^w}P(h)e^{-sG^w}=P(h)+ish(H_p G)^w+O_{L^2\to L^2}(h^2)$.
\end{prop}
\begin{proof} 1. See for example~\cite[Theorem~8.3]{e-z} (with $m(x,\xi)=1$).
The full symbol of $e^{sG^w}$ can be recovered from the evolution equation satisfied
by this family of operators; we see that it is equal to 1 outside of
a compact set.

2. It suffices to differentiate both sides of the equation in $s$, divide
them by $h$, and compare the principal symbols.
\end{proof} 

\section{Analysis near the zero energy}

In this section, we prove part 2 of Proposition~\ref{l:radial}.
Take $h>0$ such that $\Real\lambda=h^{-2}$. Put
$$
\tilde\mu=h^2\Imag\lambda,\
\tilde k=hk,\
\tilde\omega=h\omega,\
\tilde\omega_\pm=h\omega_\pm;
$$
then \eqref{e:radial-cond} implies that
\begin{equation}\label{e:semiclassical-small}
|\mu|\leq \varepsilon_r,\
|a\tilde k|\leq \varepsilon_r,\
|\tilde\omega|\leq\varepsilon_r,\
|\tilde\omega_\pm|\leq \varepsilon_r,
\end{equation}
where $\varepsilon_r>0$ and $h$ can be made arbitrarily small by
choice of $C_r$ and $\psi$. If $P_x$ is the operator
in~\eqref{e:p-x}, then $P_x=h^{-2}\tilde P_x$, where
$$
\begin{gathered}
\tilde P_x(h;\tilde\omega,\tilde\mu,\tilde k)=h^2D_x^2+\widetilde V_x(x;\tilde\omega,\tilde\mu,\tilde k),\\
\widetilde V_x(x;\tilde\omega,\tilde\mu,\tilde k)=(1+i\tilde\mu)\Delta_r-(1+\alpha)^2((r^2+a^2)\tilde\omega-a\tilde k)^2.
\end{gathered}
$$

Now, we use Proposition~\ref{l:u-complex}.
Let $u$ be an outgoing function in the sense of Definition~\ref{d:outgoing}
and assume that $f=\tilde P_x u$ is supported in $K_x$. Then $u$ satisfies~\eqref{e:u-pm}
for $|x|>X_0$ and some functions $v_\pm$.
Fix $x_+>X_0$ and consider the function
\begin{equation}\label{e:v-1}
v_1(y)=v_+(e^{-A_+(x_++iy)};\omega,\lambda,k),\
y\in \mathbb R.
\end{equation}
This is a $2\pi/A_+$-periodic function; we can think of it as a
function on the circle. It follows from the differential
equation satisfied by $v_+$ together with Cauchy-Riemann equations that
$Q(h)v_1(y)=0$, where
$$
Q(h;\tilde\omega,\tilde \mu,\tilde k)=(-ihD_y+\tilde\omega_+)^2+\widetilde V_x(x_++iy;\tilde\omega,\tilde\mu,\tilde k).
$$
Let $q(y,\eta)$ be the semiclassical symbol of $Q$:
$$
q(y,\eta)=(-i\eta+\tilde\omega_+)^2+\widetilde V_x(x_++iy).
$$
For small $h$, the function $v_1(y)$ has to be (semiclassically)
microlocalized on the
set $\{q=0\}$. Since the symbol $q$ is complex-valued, in a generic
situation this set will consist of isolated points. Also, since $v_1$
is the restriction to a certain circle of the function $v_+$, which is
holomorphic inside this circle, it is microlocalized in $\{\eta\leq
0\}$. Therefore, if the equation $q(y,\eta)=0$ has only one root with
$\eta\leq 0$, then the function $v_1$ has to be microlocalized at this
root. If furthermore $\bar q$ satisfies H\"ormander's hypoellipticity
condition, one can obtain an asymptotic decomposition of $v_1$ in
powers of $h$. We will only need a weak corollary of such
decomposition; here is a self-contained proof of the required
estimates:

\begin{prop}\label{l:vertical}
Assume that $x_+>X_0$ is chosen so that:
\begin{itemize}
\item the equation $q(y,\eta)=0$, $y\in \mathbb S^1$, has exactly one root $(y_0,\eta_0)$ such that
$\eta_0<0$;
\item the equation $q(y,\eta)=0$ has no roots with $\eta=0$;
\item the condition $i\{q,\bar q\}<0$
is satisfied at $(y_0,\eta_0)$;
\item $\Real(\eta_0+i\tilde\omega_+)<0$.
\end{itemize}
(If all of the above hold, we say that we have \textbf{vertical
control} at $x_+$ and $(y_0,\eta_0)$ is called the
\textbf{microlocalization point}.) Let $\eta(y)$ be the family of
solutions to $q(y,\eta(y))=0$ with $\eta(y_0)=\eta_0$.  Then for each
$N$, each $\chi(y,\eta)\in C_0^\infty$ that is equal to 1 near
$(y_0,\eta_0)$, and $h$ small enough, we have
\begin{gather}\label{e:vertical-1}
\|(1-\chi^w(y,hD_y))v_1\|_{H^N_h}=O(h^\infty)\|v_1\|_{L^2},\\
\label{e:vertical-2}
\|(hD_y-\eta(y))v_1\|_{H^N_h}=O(h)\|v_1\|_{L^2},\\
\label{e:vertical-3}
\|v_1\|_{L^2}\leq Ch^{1/4}|v_1(y_0)|,\\
\label{e:vertical-4}
|(hD_y-\eta_0)v_1(y_0)|\leq Ch^{1/2}\|v_1\|_{L^2},\\
\label{e:vertical-5}
\Real\bigg({h\partial_x u_+(x_++iy_0)\over u_+(x_++iy_0)}\bigg)\leq -{1\over C}<0.
\end{gather}
Similar statements are true for $u_+$
replaced by $u_-$, with the opposite inequality sign in
\eqref{e:vertical-5}.
\end{prop}
\begin{proof}
\textit{\eqref{e:vertical-1}:} We know that
$$
\inf\{\eta\mid q(y,\eta)=0,\ (y,\eta)\neq (y_0,\eta_0)\}>0.
$$
Therefore, we can decompose $1=\chi+\chi_++\chi_0$, where $\chi_+$
depends only on the $\eta$ variable, is supported in $\{\eta>0\}$, and
is equal to 1 for large positive $\eta$ and near every root of the
equation $q(y,\eta)=0$ with $\eta>0$. Since $v_+$ is holomorphic at
zero, its Taylor series provides the Fourier series for $v_1$; it then
follows from \eqref{e:fourier-series-multiplier} that
$$
\|\chi_+^w(y,hD_y)v_1\|_{H_h^N}=O(h^\infty)\|v_1\|_{L^2}.
$$
Next, the symbol $q$ is elliptic near $\supp\chi_0$; therefore, by
Proposition~\ref{l:elliptic} (whose proof applies
without changes to our case), since $Q(h)v_1=0$, we have
$$
\|\chi_0^w(y,hD_y)v_1\|_{H_h^N}=O(h^\infty)\|v_1\|_{L^2}.
$$
This finishes the proof.\smallskip

\textit{\eqref{e:vertical-2}:}
Take a small cutoff $\chi$ as above, and factor $q=(\eta-\eta(y))q_1$,
where $q_1(y,\eta)$ is nonzero near $\supp\chi$. We then find a
compactly supported symbol $r_1$ with $r_1q_1=1$ near
$\supp\chi$. Now, we have
$$
\begin{gathered}
\|\chi^w(y,hD_y)(r_1^w(y,hD_y)q_1^w(y,hD_y)-1)(hD_y-\eta(y))v_1\|_{H^N_h}=O(h)\|v_1\|_{L^2},\\
\|(1-\chi^w(y,hD_y))(r_1^w(y,hD_y)q_1^w(y,hD_y)-1)(hD_y-\eta(y))v_1\|_{H^N_h}=O(h^\infty)\|v_1\|_{L^2},\\
\|r_1^w(y,hD_y)(q_1^w(y,hD_y)(hD_y-\eta(y))-Q(h))v_1\|_{H^N_h}=O(h)\|v_1\|_{L^2}.
\end{gathered}
$$
It remains to add these up.\smallskip

\textit{\eqref{e:vertical-3}:} 
We cut off $v_1$ to make it supported in a small $\varepsilon$-neighborhood
of $y_0$. Put $f=(h\partial_y-i\eta(y))v_1$; we know that
$\|f\|_{L^2}\leq Ch\|v_1\|_{L^2}$. Now, put
$$
\Phi(y)=\int_{y_0}^y \eta(y')\,dy'.
$$
The condition $i\{q,\bar q\}|_{(y_0,\eta_0)}<0$ is equivalent to
$$
\Imag \partial_y\eta(y_0)>0;
$$
it follows that
\begin{equation}\label{e:vertical-3-ineq}
\Imag(\Phi(y)-\Phi(y'))\geq\beta ((y-y_0)^2-(y'-y_0)^2)
\end{equation}
for some $\beta>0$, $|y-y_0|<\varepsilon$, and $y'$ between $y$ and $y_0$.
(To see that,
represent the left-hand side as an integral.)
Now,
$$
v_1(y)=e^{i\Phi(y)/h}v_1(y_0)+h^{-1}\int_{y_0}^y e^{i(\Phi(y)-\Phi(y'))/h}f(y')\,dy'.
$$
Let $Tf(y)$ be the second term in the sum above; it suffices to prove that
$$
\|Tf\|_{L^2(y_0-\varepsilon,y_0+\varepsilon)}\leq Ch^{-1/2}\|f\|_{L^2(y_0-\varepsilon,y_0+\varepsilon)}.
$$
This can be reduced to the inequalities
$$
\begin{gathered}
\sup_{0\leq y-y_0<\varepsilon} \int_{y_0}^y |e^{i(\Phi(y)-\Phi(y'))/h}|\,dy'=O(h^{1/2}),\\
\sup_{0\leq y'-y_0<\varepsilon} \int_{y'}^{y_0+\varepsilon} |e^{i(\Phi(y)-\Phi(y'))/h}|\,dy=O(h^{1/2}).
\end{gathered}
$$
and similar inequalities for the case $y,y'<y_0$.  We now
use~\eqref{e:vertical-3-ineq}; after a change of variables, it
suffices to prove that
$$
\begin{gathered}
\sup_{y>0}\int_0^y e^{(y')^2-y^2}\,dy'<\infty,\
\sup_{y'>0}\int_{y'}^\infty e^{(y')^2-y^2}\,dy<\infty.
\end{gathered}
$$
To prove the first of these inequalities, make the change of variables
$y'=ys$; then the integral becomes
$$
\int_0^1 ye^{y^2(s^2-1)}\,ds.
$$
However, $ye^{y^2(s^2-1)}\leq C(1-s^2)^{-1/2}$, and the integral
of the latter converges.

After the change of variables $y=y'+s$, the integral of the second inequality above becomes
$$
\int_0^\infty e^{-2y's-s^2}\,ds.
$$
This can be estimated by $\int e^{-s^2}\,ds$.\smallskip

\textit{\eqref{e:vertical-4}:} Let $\chi\in C_0^\infty(\mathbb R^2)$ have $\chi=1$ near $(y_0,\eta_0)$.
Combining~\eqref{e:l2-l-infty} and~\eqref{e:vertical-2} with
$$
\|(1-\chi^w(y,hD_y))(hD_y-\eta(y))v_1\|_{L^\infty}=O(h^\infty)\|v_1\|_{L^2},
$$
we get $\|(hD_y-\eta(y))v_1\|_{L^\infty}=O(h^{1/2})\|v_1\|_{L^2}$;
it remains to take $y=y_0$.

\textit{\eqref{e:vertical-5}:} Follows immediately from \eqref{e:vertical-3},
\eqref{e:vertical-4}, \eqref{e:v-1}, Cauchy-Riemann equations, and
the fact that $\Real(\eta_0+i\tilde\omega_+)<0$.
\end{proof}

If $\tilde P_x$ were a semiclassical Schr\"odinger operator with a strictly
positive potential, then a standard integration by parts argument
would give us $\|u\|_{L^2}\leq C\|\tilde P_xu\|_{L^2}$ on any interval for
each function $u$ satisfying the condition~\eqref{e:vertical-5} at the
right endpoint of this interval and the opposite condition at its left
endpoint.  We now generalize this argument to our case.  Assume that
we have vertical control at the points $x_\pm$, $\pm x_\pm>X_0$, and
let $(y_\pm,\eta_\pm)$ be the corresponding microlocalization points.
Let $\gamma$ be a contour in the $z$ plane; we say that we have
\textbf{horizontal control} on $\gamma$ if:
\begin{itemize}
\item $\gamma\cap \mathbb \{|\Real z|\leq X_0\}\subset \mathbb R$;
\item the endpoints of $\gamma$ are $z_\pm=x_\pm+iy_\pm$;
\item $\gamma$ is given by $\Imag z=F(\Real z)$, where $F$ is a smooth function and $F'(x_\pm)=0$;
\item $\Real[(1+iF'(x))\widetilde V_x(x+iF(x))]\geq {1\over C_1}> 0$ for all $x$.
\end{itemize}

Now, let $u$ be as in the beginning of this section
and define $u_\gamma(x)$, $x_-\leq x\leq x_+$,
by Proposition~\ref{l:u-complex}. Then $\tilde P_\gamma u_\gamma=f$, where
$$
\tilde P_\gamma=\bigg({1\over 1+iF'(x)}hD_x\bigg)^2+\widetilde V_x(x+iF(x)).
$$
\begin{figure}
\includegraphics{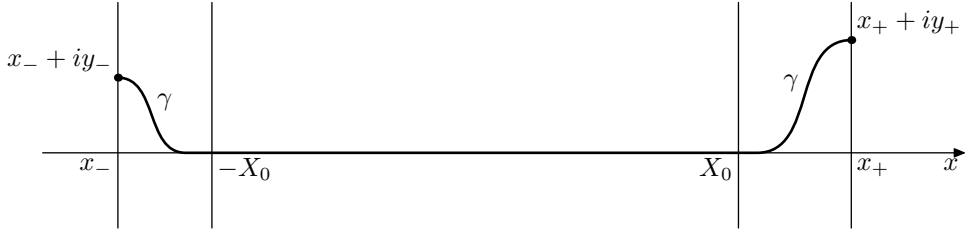}
\caption{A contour with horizontal control}
\end{figure}
If we have vertical control at the endpoints of $\gamma$, then
by~\eqref{e:vertical-5},
$$
\pm\Real(\overline{u_\gamma(x_\pm)}h\partial_x u_\gamma(x_\pm))\leq -|u_\gamma(x_\pm)|^2/C<0.
$$
Now, assume that we have horizontal control on $\gamma$.
Then we can integrate by parts to get
\begin{equation}\label{e:ibp}
\begin{gathered}
\int_{x_-}^{x_+} \Real (\overline{u_\gamma}(1+iF'(x))f)\,dx=
\int_{x_-}^{x_+} \Real (\overline{u_\gamma}\cdot (1+iF'(x))\tilde P_\gamma u_\gamma)\,dx\\
= \int_{x_-}^{x_+} \Real {|hD_xu_\gamma |^2\over 1+iF'(x)}\,dx
+\int_{x_-}^{x_+} \Real[(1+iF'(x))\widetilde V_x(x+iF(x))]\cdot |u_\gamma|^2\,dx\\
-h^2\Real(\overline{u_\gamma}\partial_x u_\gamma)|_{x=x_-}^{x_+}
\geq {1\over C_1}(\|u_\gamma\|_{L^2}^2+h(|u_\gamma(x_+)|^2+|u_\gamma(x_-)|^2)).
\end{gathered} 
\end{equation}
Therefore,
$$
\|u_\gamma\|_{L^2}\leq C\|f\|_{L^2}.
$$
It follows that the operator $R_x$ from~\eqref{e:r-x} is correctly defined and
$$
\|1_{K_x}R_x 1_{K_x}\|_{L^2\to L^2}\leq Ch^2,
$$
This proves the estimate~\eqref{e:radial-est} under the assumptions made above.

We now prove~\eqref{e:radial-est-1.5}. We concentrate on
the estimate on $K_+$; the case of $K_-$ is considered in a similar fashion.
First of all, it follows from~\eqref{e:ibp} that
\begin{equation}\label{e:exponential-est-1}
|u_\gamma(x_+)|\leq Ch^{-1/2}\|f\|_{L^2}.
\end{equation}
Now, assume that we have vertical control at every point of the interval
$I_+=[x_+,x_++1]$ and let $(y(x),\eta(x))$, $x\in I_+$, be the corresponding
microlocalization points. Let $v_x(z)=e^{-i\omega_+ z}u(z)$ and 
put $v_2(x)=v_x(x+iy(x))$; then
\begin{equation}\label{e:exponential-est-2}
|v_2(x_+)|\leq Ce^{\Imag(\omega_+ z_+)}|u_\gamma(x_+)|.
\end{equation}
Now, by Proposition~\ref{l:vertical}, we have
\begin{equation}\label{e:exponential-est-3}
|h \partial_x \ln |v_2(x)|-\eta(x)|\leq Ch^{3/4},\
x\in I_+.
\end{equation}
Integrating~\eqref{e:exponential-est-3} and combining it with~\eqref{e:exponential-est-1}
and~\eqref{e:exponential-est-2}, we see that if
\begin{equation}\label{e:exponential-integral-cond}
\Imag(\tilde\omega_+z_+)+\int_{x_+}^{x_++1}\eta(x)\,dx<-2\delta_0
\end{equation}
for some $\delta_0>0$, then
$
|v_2(x_++1)|\leq C e^{-\delta_0/h}\|f\|_{L^2}
$.
Next, $v_2(x_++1)$ is the value of $v$ at the microlocalization point;
therefore, by Proposition~\ref{l:vertical} and~\eqref{e:l2-l-infty},
$$
\sup_{y\in \mathbb R} |v_x(x_++1+iy)|\leq Ch^{-1/4}e^{-\delta_0/h}\|f\|_{L^2}. 
$$
Finally, recall that $v_x(z)=v_w(e^{-A_+z})$, where the function
$v_w(w)$ is holomorphic inside the disc $B_w=\{|w|\leq e^{-A_+(x_++1)}\}$.
The change of variables $w\to r$ is holomorphic by Proposition~\ref{e:regge-wheeler};
let $K_+^w$ be the image of $K_+$ under this change of variables. If $\delta_{r0}$
is small enough, then $K_+^w$ lies in the interior of $B_w$; then by the maximum principle
and Cauchy estimates on derivatives, we can estimate $\|v_w\|_{C^N(K_+^w)}$
for each $N$ by $O(h^\infty)\|f\|_{L^2}$. This completes the proof of~\eqref{e:radial-est-1.5}
if the conditions above are satisfied.

To prove part~2 of Proposition~\ref{l:radial}, it remains to establish
both vertical and horizontal control in our situation:
\begin{prop}\label{l:geometric}
Assume that $\delta_r>0$.
Then there exist $\varepsilon_r$ and $x_\pm$, $\pm x_\pm>X_0$, such that
under the conditions~\eqref{e:semiclassical-small}, 
\begin{itemize}
\item we have vertical control at every point of the intervals $I_+=[x_+,x_++1]$
and $I_-=[x_--1,x_-]$;
\item we have horizontal control on a certain contour $\gamma$;
\item the inequality~\eqref{e:exponential-integral-cond} (and its analogue on $I_-$)
holds.
\end{itemize}
\end{prop}
\begin{proof}
Let us first assume that $a\tilde k=\tilde\mu=\tilde\omega=0$. Then $\tilde\omega_\pm=0$
and $q(y,\eta)=-\eta^2+\Delta_r(x_++iy)$.  Therefore, if we choose
$x_+$ large enough, there exists exactly one solution $(y_0,\eta_0)$
to the equation $q(y,\eta)=0$ with $\eta\leq 0$, and this solution has
$y_0=0$. It is easy to verify that in that case we have vertical
control on $I_+$. Similarly one can choose the point $x_-$; moreover,
we can assume that $K_r\subset (x_-,x_+)$ after the change of
variables $r\to x$. Next, since $\widetilde V_x=\Delta_r$, we can take
$\gamma$ to be the interval~$[x_-,x_+]$ of the real line. The
condition~\eqref{e:exponential-integral-cond} holds because $\eta(x)<0$
for every $x$ and $\tilde\omega_\pm=0$.

Now, fix $x_\pm$ as above.  The parameters of our problem are $a$,
varying in a compact set, $\Lambda$ and $M$, both fixed, and
$a\tilde k,\tilde\mu,\tilde\omega$. By the implicit function theorem, if the last
three parameters are small enough, the (open) conditions of vertical
control and the condition~\eqref{e:exponential-integral-cond}
are still satisfied, yielding $y_\pm$ close to zero. Then one
can take the contour $\gamma$ defined by $\Imag z=F(\Real z)$, where
$F=0$ near $K_r$, $F(x_\pm)=y_\pm$, and $F$ is small in $C^\infty$.
For small values of $a\tilde k,\tilde\mu,\tilde\omega$, we will still have horizontal
control on this $\gamma$, proving the proposition.
\end{proof}

\section{Resonance free strip}

In this section, we prove Theorem~\ref{t:resonance-free-strip}.
First of all, by Proposition~\ref{l:smart-contour}, it suffices to prove
\begin{prop}\label{l:resonance-free-radial}
Fix $\delta_r>0$, $\varepsilon_e>0$, and a large constant $C'$.
Then there exist constants $a_0>0$ and $C''$ such that for
$$
\begin{gathered}
|\Real\lambda|+k^2\leq C'|\Real\omega|^2,\
|a|<a_0,\
|\Real\omega|\geq 1/C'',\\
|\Imag\omega|\leq 1/C'',\
|\Imag\lambda|\leq |\Real\omega|/C''
\end{gathered}
$$
we have
$$
\|1_{K_r}R_r(\omega,\lambda,k)1_{K_r}\|_{L^2\to L^2}\leq C''|\omega|^{\varepsilon_e-1}.
$$
\end{prop}
Indeed, we take $C'$ large enough so that $C_k^2(1+|\omega|)^2+L\leq C'|\omega|^2/2$;
then, we put $l_1=L$ and $l_2=|\Real\omega|/C''$.

Next, we reformulate Proposition~\ref{l:resonance-free-radial} in semiclassical
terms. Without loss of generality, we may assume that $\Real\omega>0$. Put
$h=(\Real\omega)^{-1}$ and consider the rescaled operator
$$
\tilde P_x=h^2P_x=h^2D_x^2+(\tilde\lambda+ih\tilde\mu)\Delta_r-(1+\alpha)^2((r^2+a^2)(1+ih\nu)-a\tilde k)^2.
$$
Here $P_x$ is the operator in~\eqref{e:p-x} and
$$
\tilde\lambda=h^2\Real\lambda,\
\tilde k=hk,\
\tilde\mu=h\Imag\lambda,\
\nu=\Imag\omega.
$$
Then it suffices to prove that for $h$ small enough and under the conditions
\begin{equation}\label{e:semiclassical-cond}
|\tilde\lambda|\leq C',\
|\tilde k|\leq C',\
|\tilde \mu|\leq 1/C'',\
|\nu|\leq 1/C'',
\end{equation}
for each $f(x)\in L^2\cap \mathcal E'(K_x)$ and solution $u(x)$
to the equation $\tilde P_x u=f$ which is outgoing in the sense of Definition~\ref{d:outgoing},
we have
\begin{equation}\label{e:semiclassical-res}
\|u\|_{L^2(K_x)}\leq Ch^{-1-\varepsilon_e}\|f\|_{L^2}.
\end{equation}
(Here $K_x$ is the image of $K_r=(r_-+\delta_r,r_+-\delta_r)$ under the change of variables $r\to x$.)
We write $\tilde P_x=h^2D_x^2+\widetilde V_0+ih\widetilde V_1$, where
$$
\begin{gathered}
\widetilde V_0=\tilde\lambda\Delta_r-(1+\alpha)^2(r^2+a^2-a\tilde k)^2,\\
\widetilde V_1=\tilde\mu\Delta_r-\nu(1+\alpha)^2(r^2+a^2)((r^2+a^2)(2+ih\nu)-2a\tilde k)^2.
\end{gathered}
$$
We note that $\widetilde V_0$ is real-valued and
$\|\widetilde V_1\|_{L^\infty}\leq C/C''$
for some global constant $C$.

\smallskip

We now apply the method of complex scaling.
(This method was first developed by Aguilar and Combes in~\cite{a-c}; see~\cite{s-z} and
the references there for more recent developments.)
Consider the contour $\gamma$ in the complex plane given by
$\Imag x=F(\Real x)$, with $F$ defined by
\begin{equation}\label{e:complex-scaling-f}
F(x)=\begin{cases}
0,&|x|\leq R;\\
F_0(x-R),&x\geq R;\\
-F_0(-x-R),&x\leq -R.
\end{cases}
\end{equation}
Here $R>X_0$ is large and $F_0\in C_0^\infty(0,\infty)$ is a fixed function
such that $F'_0\geq 0$ and $F''_0\geq 0$ for all $x$ and $F'_0(x)=1$ for $x\geq 1$.
(We could use a contour which forms an arbitrary fixed angle $\tilde\theta\in (0,\pi/2)$ with
the horizonal axis for large $x$; we choose the angle $\pi/4$ to simplify the
formulas.)
\begin{figure}
\includegraphics{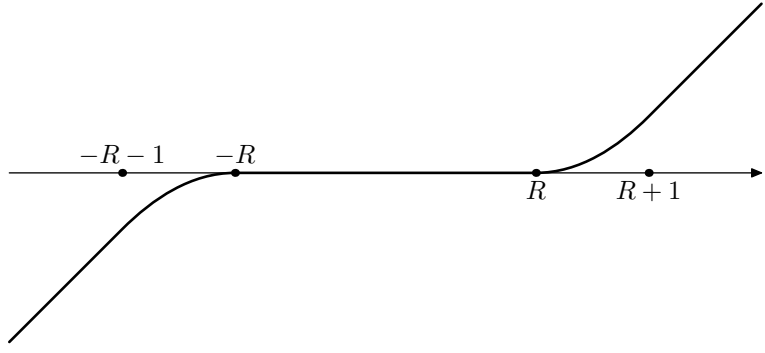}
\caption{The contour used for complex scaling}
\end{figure}
Now, let $u$ be an outgoing solution to the equation $\tilde P_x u=f\in L^2\cap \mathcal E'(K_x)$,
as above.
By Proposition~\ref{l:u-complex}, we can define the restriction $u_\gamma$ of $u$
to $\gamma$ and $\tilde P_\gamma u_\gamma=f$, where
$$
\tilde P_\gamma=\bigg({h\over 1+iF'(x)}D_x\bigg)^2+\widetilde V_0(x+iF(x))+ih\widetilde V_1(x+iF(x)).
$$
Also, for $a$ and $h$ small enough, $u_\gamma$ lies in $H^2(\mathbb R)$.
Therefore, in order to prove~\eqref{e:semiclassical-res}, it is enough to show
that for each $u_\gamma\in H^2(\mathbb R)$, we have
\begin{equation}\label{e:semiclassical-res2}
\|u_\gamma\|_{L^2(\mathbb R)}\leq Ch^{-1-\varepsilon_e}\|\tilde P_\gamma u_\gamma\|_{L^2(\mathbb R)}.
\end{equation}
Let $p_0$ and $p_{\gamma 0}$ be the semiclassical principal symbols of $\tilde P_x$ and
$\tilde P_\gamma$:
$$
\begin{gathered}
p_0(x,\xi)=\xi^2+\widetilde V_0(x),\\
p_{\gamma 0}(x,\xi)={\xi^2\over (1+iF'(x))^2}+\widetilde V_0(x+iF(x)).
\end{gathered}
$$
The key property of the operator $\tilde P_\gamma$, as opposed to $\tilde P_x$,
is ellipticity at infinity, which follows from
the fact that $\widetilde V_0(\pm\infty)=-\tilde\omega_{0\pm}^2$, where
$$
\tilde\omega_{0\pm}=(1+\alpha)(r_\pm^2+a^2-a\tilde k)\geq 1/C>0
$$
if $a$ is small enough. Certain other properties of the symbol $p_{\gamma 0}$
can be derived using only the behavior of $\widetilde V_0$ near infinity given by~\eqref{e:v-x-complex};
we state them for a general class of potentials:
\begin{prop}\label{l:complex-scaling}
Assume that $V(x)$, $x>0$, is a real-valued potential such that
for $x>X_0$, we have
$V(x)=V_+(e^{-A_+x})$ for a certain constant $A_+>0$
and a function $V_+(w)$ holomorphic in $\{|w|<e^{-A_+X_0}\}$;
assume also that $V_+(0)<0$. Let $F(x)$ be as in~\eqref{e:complex-scaling-f}, for $R>X_0$, and put
$$
\begin{gathered}
p(x,\xi)=\xi^2+V(x),\\
p_\gamma(x,\xi)={\xi^2\over (1+iF'(x))^2}+V(x+iF(x)).
\end{gathered}
$$
Then there exists a constant $C_c$ such that 
for $R$ large enough and $\delta>0$ small enough,
\begin{gather}
\label{e:complex-scaling-1}
\text{if }x\geq R+1,\text{ then  }|p_\gamma(x,\xi)|\geq 1/C_c>0,\\
\label{e:complex-scaling-2}
\text{if }|p_\gamma(x,\xi)|\leq e^{-A_+R},\text{ then }\Imag p_\gamma(x,\xi)\leq 0,\\
\label{e:complex-scaling-3}
\text{if }|p_\gamma(x,\xi)|\leq\delta,\text{ then }|p(x,\xi)|\leq C_c\delta,\
|\nabla (\Real p_\gamma-p)(x,\xi)|\leq C_c\delta.
\end{gather}
Similar facts hold if $V$ is defined on $x<0$ instead.
\end{prop}
\begin{proof} Without loss of generality, we assume that $A_+=1$ and $V_+(0)=-1$. 
First of all, if $x\geq R+1$, then
$$
p_\gamma(x,\xi)=-i\xi^2/2+V(x+iF(x))=-i\xi^2/2-1+O(e^{-R}).
$$
For $R$ large enough, we then get $|p_\gamma(x,\xi)|\geq 1/2$, thus
proving~\eqref{e:complex-scaling-1}.

For the rest of the proof, we may assume that $R\leq x\leq R+1$. Then,
since $F'_0$ is increasing, we get $0\leq F(x)\leq F'(x)$.
Suppose that $|p_\gamma(x,\xi)|\leq \delta$; then
\begin{equation}\label{e:complex-scaling-4}
\begin{gathered}
{\xi^2\over (1+iF'(x))^2}=-V(x+iF(x))+O(\delta)\\
=-V(x)(1+O(\delta+e^{-R}F(x)))=1+O(\delta+e^{-R}).
\end{gathered}
\end{equation}
Taking the arguments of both sides, we get
$$
F'(x)\leq C(\delta+e^{-R}F(x))\leq C\delta+Ce^{-R}F'(x).
$$
Then for $R$ large enough,
$$
|p_\gamma(x,\xi)|\leq\delta\to F'(x)\leq C\delta.
$$
This proves~\eqref{e:complex-scaling-3}, if we note that $F''$ is bounded and
$$
\Real p_\gamma(x,\xi)-p(x,\xi)=\xi^2 G_1(F'(x)^2)+G_2(F(x),x)
$$
for certain smooth functions $G_1$ and $G_2$ that are equal to
zero at $F'=0$ and $F=0$, respectively.

Now, putting $\delta=e^{-R}$ and taking the arguments and then the absolute values
of both sides of~\eqref{e:complex-scaling-4}, we get for $|p_\gamma|\leq\delta$,
$$
F'(x)=O(e^{-R}),\
\xi^2=1+O(e^{-R}).
$$
Therefore,
$$
\Imag p_\gamma(x,\xi)=-2F'(x)+O(e^{-R}(F(x)+F'(x)))=F'(x)(-2+O(e^{-R})),
$$
which proves~\eqref{e:complex-scaling-2}.
\end{proof} 

Now, we study the trapping properties of the Hamiltonian flow of $p_0$ at the zero
energy:
\begin{prop}\label{l:radial-nontrapping}
There exist constants $C_V$ and $\delta_V$ such that for $a$ small enough
and every $\tilde\lambda$, $\tilde k$
satisfying~\eqref{e:semiclassical-cond}, at least one of the three dynamical cases below holds:
\begin{enumerate}
\item $\widetilde V_0\leq-\delta_V$ everywhere;
\item $\{|\widetilde V_0|\leq \delta_V\}=[x_1,x_2]\sqcup [x_3,x_4]$, where
$-C_V\leq x_1<x_2<x_3<x_4\leq C_V$ and $\widetilde V_0'\geq 1/C_V$ on $[x_1,x_2]$,
$\widetilde V_0'\leq -1/C_V$ on $[x_3,x_4]$;
\item $\{\widetilde V_0\geq -\delta_V\}=[x_1,x_2]$ with $|x_j|\leq C_V$,
$\widetilde V_0''\leq -1/C_V$ on $[x_1,x_2]$.
\end{enumerate}
\end{prop}
\begin{proof}
\begin{figure}
\includegraphics{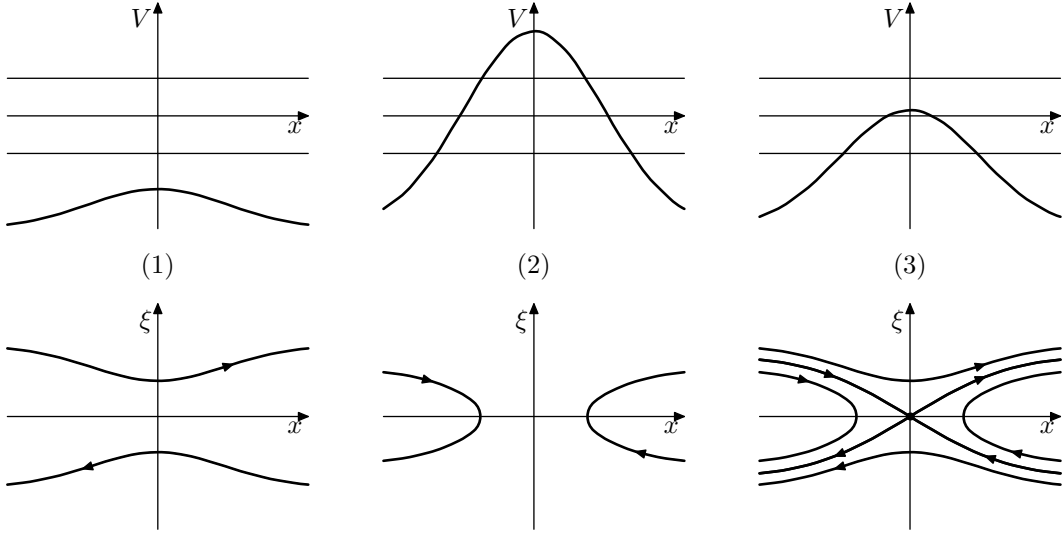}
\caption{Three cases for the potential $\widetilde V_0$ and its Hamiltonian flow
near the zero energy. The horizontal lines correspond to $\widetilde V_0=\pm\delta_V$.}
\end{figure} First of all, if $\tilde\lambda$ is small enough or $\tilde\lambda<0$, then
we have $\widetilde V_0<0$ everywhere and therefore case (1) holds for $\delta_V$ small enough.
Therefore, we may assume that $1/C\leq \tilde\lambda\leq C$ for some constant $C$.
Now, we write
$$
\begin{gathered}
\widetilde V_0(x)=G_V(r)(F_V(r)-\tilde\lambda^{-1}),\\
G_V(r)=\tilde\lambda(1+\alpha)^2(r^2+a^2-a\tilde k)^2,\\
F_V(r)={\Delta_r\over (1+\alpha)^2(r^2+a^2-a\tilde k)^2}.
\end{gathered}
$$
Note that $1/C\leq G_V(r)\leq C$ for $a$ small enough,
some constant $C$, and all $r$.
As for $F_V$, there exists $\varepsilon>0$ such that for $a$ small enough,
$\partial_r F_V(r)\geq 1/C>0$ for $r\leq 3M-\varepsilon$,
$\partial_r F_V(r)\leq -1/C<0$ for $r\geq 3M+\varepsilon$, and
$\partial_r^2 F_V(r)\leq -1/C<0$ for $|r-3M|\leq\varepsilon$. Indeed,
this is true for $a=0$ and follows for small $a$ by a perturbation argument.
Let $r_0\in [3M-\varepsilon,3M+\varepsilon]$ be the point where $F_V$ achieves its maximal value.
Take small $\delta_1>0$; then we have one of the following three cases, each of which in turn
implies the corresponding case in the statement of this proposition:
\begin{enumerate}
\item $F_V(r_0)-\tilde\lambda^{-1}\leq -\delta_1$. Then $\widetilde V_0(x)<-\delta_V$ for all $x$ and
$\delta_V>0$ small enough.
\item $F_V(r_0)-\tilde\lambda^{-1}\geq\delta_1$.
Then for $\delta_2<\delta_1/2$,
$\{|F_V-\tilde\lambda^{-1}|\leq\delta_2\}=[x_1,x_2]\sqcup [x_3,x_4]$, where
$x_2<x_3$, $x_j$ are bounded by a global constant (since $\tilde\lambda$ is bounded from above),
and $\partial_r F_V(r)>1/C_\delta>0$ for $x\in [x_1,x_2]$, $\partial_r F_V(r)<-1/C_\delta<0$ for
$x\in [x_3,x_4]$. Here $C_\delta$ is a constant depending on $\delta_1$,
but not on $\delta_2$. It follows that for $\delta_2$ small enough depending on $\delta_1$, we have
$\widetilde V'_0(x)>0$ for $x\in [x_1,x_2]$ and $\widetilde V'_0(x)<0$ for $x\in [x_3,x_4]$;
also, for $\delta_V$ small enough, we have $\{|\widetilde V_0|\leq\delta_V\}\subset [x_1,x_2]\sqcup [x_3,x_4]$.
\item $|F_V(r_0)-\tilde\lambda^{-1}|<\delta_1$. Then $\{F_V-\tilde\lambda^{-1}>-\delta_1\}=[x_1,x_2]$
with $\partial_r^2 F_V(r)<-1/C<0$ for $x\in [x_1,x_2]$. For $\delta_1$ small enough, we then get
$\widetilde V''_0<-1/C<0$ for $x\in [x_1,x_2]$, and for $\delta_V$ small enough, we have
$\{\widetilde V_0\geq -\delta_V\}\subset [x_1,x_2]$.\qedhere
\end{enumerate}
\end{proof}

We are now ready to prove~\eqref{e:semiclassical-res2} and, therefore, Theorem~\ref{t:resonance-free-strip}.
Fix $R$ large enough so that Proposition~\ref{l:complex-scaling} holds.
The first two cases are in Proposition~\ref{l:radial-nontrapping} are
nontrapping; it follows that there exists
an escape function $G\in C_0^\infty(\mathbb R^2)$ such that
$H_{p_0} G<0$ on $\{|p_0|\leq\delta_V/2\}\cap \{|x|\leq R+2\}$. In the third case, we have
hyperbolic trapping with the trapped set consisting of a single
point $(x_0,0)$, where $x_0$ is the point where $\widetilde V_0$
achieves its maximal value; therefore, there still exists an escape function
$G\in C_0^\infty(\mathbb R^2)$ such that $H_{p_0}G\leq 0$ on $\{|p_0|\leq\delta_V/2\}\cap \{|x|\leq R+2\}$
and $H_{p_0}G<0$ on
$\{|p_0|\leq\delta_V/2\}\cap \{|x|\leq R+2\}\setminus U(x_0,0)$, where $U$ is
a neighborhood of $(x_0,0)$ which can be made arbitrarily small by the choice of $G$
(see~\cite[Proposition~A.6]{g-s}). 
Now, given Proposition~\ref{l:complex-scaling}, we can choose $\delta_0>0$
such that
\begin{equation}\label{e:complex-scaling-1.1}
\Imag p_{\gamma 0}\leq 0\text{ on }\{|p_{\gamma 0}|\leq \delta_0\}
\end{equation}
and for cases (1) and (2) of Proposition~\ref{l:radial-nontrapping}, we have
\begin{equation}\label{e:complex-scaling-1.2}
H_{\Real p_{\gamma 0}}G\leq -1/C<0\text{ on }\{|p_{\gamma 0}|\leq \delta_0\},
\end{equation}
and for case (3) of Proposition~\ref{l:radial-nontrapping}, we have
\begin{equation}\label{e:complex-scaling-1.3}
\begin{gathered}
H_{\Real p_{\gamma 0}}G\leq 0\text{ on }\{|p_{\gamma 0}|\leq\delta_0\},\\
H_{\Real p_{\gamma 0}}G\leq -1/C<0\text{ on }\{|p_{\gamma 0}|\leq\delta_0\}\setminus U(x_0,0).
\end{gathered}
\end{equation}

Armed with these inequalities, we can handle the nontrapping cases even without requiring
that $\mu$ and $\nu$ be small. The statement below follows the method initially developed
in~\cite{m} and is a special case of the results in~\cite[Chapter~6]{d}; however, we choose
to present the proof in our simple case:
\begin{prop}\label{l:nontrapping}
Assume that either case (1) or case (2) of Proposition~\ref{l:radial-nontrapping} holds.
Then for $\tilde\lambda$ and $\tilde k$ bounded by $C'$, $\tilde\mu$ and $\nu$ bounded by
some constant, and $h$ small enough, we have
\begin{equation}\label{e:nontrapping-est}
\|u_\gamma\|_{L^2}\leq Ch^{-1}\|\tilde P_\gamma u_\gamma\|_{L^2}
\end{equation}
for each $u_\gamma\in H^2(\mathbb R)$.
\end{prop}
\begin{proof}
Take $\chi\in C_0^\infty(\mathbb R^2)$ such that
$\supp\chi\subset \{|p_{\gamma 0}|<\delta_0\}$, but $\chi=1$ near $\{p_{\gamma 0}=0\}$.
Next, take $s>0$, to be chosen later, and put
$$
\tilde P_{\gamma,s}=e^{s G^w}\tilde P_\gamma e^{-s G^w},\
u_{\gamma,s}=e^{s G^w}\chi^w u_\gamma.
$$
Take $\chi_1\in C_0^\infty(\mathbb R^2)$ supported in $\{|p_{\gamma 0}|<\delta_0\}$, but such that
$\chi_1=1$ near $\supp\chi$. Then by part~1 of Proposition~\ref{l:exponentiation}
and~\eqref{e:semiclassical-disjoint-support},
\begin{equation}\label{e:nontrapping-est-2}
\|(1-\chi_1^w)u_{\gamma,s}\|=O(h^\infty)\|u_\gamma\|.
\end{equation}
(In the proof of the current proposition, as well as the next one, we
only use $L^2$ norms.)
Also, for some $s$-dependent constant $C$,
$$
C^{-1}\|\chi^w u_\gamma\|\leq\|u_{\gamma,s}\|\leq C\|u_\gamma\|.
$$
Now, by part~2 of Proposition~\ref{l:exponentiation}, we have
$$
\tilde P_{\gamma,s}=\tilde P_{\gamma 0}+ihV_1+ish(H_{p_{\gamma,0}}G)^w+O(h^2).
$$
Here $\tilde P_{\gamma 0}$ is the principal part of $\tilde P_\gamma$ (without $V_1$)
and the constant in $O(h^2)$ depends on $s$. We then have
$$
\begin{gathered}
\Imag(\tilde P_{\gamma,s}\chi_1^w u_{\gamma,s},\chi_1^w u_{\gamma,s})
=\Imag(\tilde P_{\gamma 0}\chi_1^w u_{\gamma,s},\chi_1^w u_{\gamma,s})
+h\Real(V_1\chi_1^w u_{\gamma,s},\chi_1^w u_{\gamma,s})\\
+sh((H_{\Real p_{\gamma,0}}G)^w \chi_1^w u_{\gamma,s},\chi_1^w u_{\gamma,s})
+O(h^2)\|\chi_1^w u_{\gamma,s}\|^2.
\end{gathered}
$$
By~\eqref{e:complex-scaling-1.1} and part~1 of Proposition~\ref{l:garding},
$$
\Imag(\tilde P_{\gamma 0}\chi_1^w u_{\gamma,s},\chi_1^w u_{\gamma,s})
\leq C h\|\chi_1^w u_{\gamma,s}\|^2+O(h^\infty)\|u_\gamma\|^2.
$$
Next, by~\eqref{e:complex-scaling-1.2} and part~2 of Proposition~\ref{l:garding},
$$
((H_{\Real p_{\gamma,0}}G)^w\chi_1^w u_{\gamma,s},\chi_1^w u_{\gamma,s})
\leq -C^{-1} \|\chi_1^w u_{\gamma,s}\|^2+O(h^\infty)\|u_\gamma\|^2.
$$
Adding these up, we get
$$
\Imag(\tilde P_{\gamma,s}\chi_1^w u_{\gamma,s},\chi_1^w u_{\gamma,s})
\leq -h(C_1^{-1}s-C_1-O(h))\|\chi_1^w u_{\gamma,s}\|^2
+O(h^\infty)\|u_\gamma\|^2.
$$
Here the constants in $O(\cdot)$ depend on $s$, but the constant $C_1$ does not.
Therefore, if we choose $s$ large enough and $h$-independent,
then for small $h$ we have the estimate
$$
\|\chi_1^w u_{\gamma,s}\|^2\leq Ch^{-1}\|\tilde P_{\gamma,s}\chi_1^w u_{\gamma,s}\|
\cdot \|\chi_1^w u_{\gamma,s}\|+O(h^\infty)\|u_\gamma\|^2.
$$
Together with~\eqref{e:nontrapping-est-2}, this gives
$$
\|\chi^w u_\gamma\|^2\leq Ch^{-1}\|\tilde P_\gamma u_\gamma\|\cdot \|u_\gamma\|
+Ch^{-1}\|[\tilde P_\gamma,\chi^w]u_\gamma\|\cdot\|u_\gamma\|+O(h^\infty)\|u_\gamma\|^2.
$$
Applying Proposition~\ref{l:elliptic} to estimate $(1-\chi^w)u_\gamma$ and
the commutator term above, we get the estimate~\eqref{e:nontrapping-est}.
\end{proof}
\noindent\textbf{Remark.} The method described above can actually be used
to obtain a \textbf{logarithmic} resonance free region; however, since we expect
the resonances generated by trapping to lie asymptotically on a lattice as
in~\cite{sb-z}, we only go a fixed amount deep into the complex plane.

The third case in Proposition~\ref{l:radial-nontrapping}
is where trapping occurs, and we analyse it as in~\cite{w-z}:
(See also~\cite{b-f-r-z} for a different method of solving the same problem.)
\begin{prop}\label{l:trapping}
Assume that case (3) in Proposition~\ref{l:radial-nontrapping} holds, and fix
$\varepsilon_e>0$.
Then for $\tilde\lambda$ and $\tilde k$ bounded by $C'$ and for
$\tilde\mu,\nu,h$ small enough, we have
\begin{equation}\label{e:trapping-est}
\|u_\gamma\|_{L^2}\leq Ch^{-1-\varepsilon_e} \|\tilde P_\gamma u_\gamma\|_{L^2}
\end{equation}
for each $u_\gamma\in H^2(\mathbb R)$.
\end{prop}
\begin{proof} First, we establish~\cite[Lemma~4.1]{w-z} in our case.
Let $x_0$ be the point where $\widetilde V_0$ achieves its maximum value.
We may assume that $|p_0(x_0,0)|=|\widetilde V_0(x_0)|<\delta_0/2$; otherwise,
we are in one of the two nontrapping cases.
Put
$$
\tilde\xi(x)=\sgn(x-x_0)\sqrt{\widetilde V_0(x_0)-\widetilde V_0(x)};
$$
since $\widetilde V''_0(x_0)<0$, it is a smooth function. Then, define the functions
$\varphi_\pm(x,\xi)=\xi\mp\tilde\xi(x)$.
We have
$$
H_{p_0}\varphi_\pm(x,\xi)=\mp c(x,\xi) \varphi_\pm(x,\xi),
$$
where $c(x,\xi)=2 \partial_x\tilde\xi(x)$
is greater than zero near the trapped point $(x_0,0)$. Also,
$\{\varphi_+,\varphi_-\}=c(x,\xi)$.
Next, take $\tilde h>h$ and large $C_0>0$, let $\chi_0\geq 0$ be supported in a small neighborhood
of $(x_0,0)$ with $\chi_0=1$ near this point, and define the modified escape function~\cite[(4.6)]{w-z}
$$
G_1(x,\xi)=-\chi_0(x,\xi)\log{\varphi_-^2(x,\xi)+h/\tilde h\over
\varphi_+^2(x,\xi)+h/\tilde h}+C_0\log(1/h)G(x,\xi).
$$
Here $G$ is an escape function satisfying~\eqref{e:complex-scaling-1.3}.
We can write
\begin{equation}\label{e:wz-h-p}
\begin{gathered}
H_{\Real p_{\gamma,0}} G_1=-2\chi_0 c\bigg({\varphi_-^2\over\varphi_-^2+h/\tilde h}
+{\varphi_+^2\over\varphi_+^2+h/\tilde h}\bigg)\\
-(H_{p_0}\chi_0)\log{\varphi_-^2+h/\tilde h\over\varphi_+^2+h/\tilde h}
+C_0\log(1/h)H_{\Real p_{\gamma 0}} G(x,\xi).
\end{gathered}
\end{equation}
Take $\chi_1$ supported in $\{|p_{\gamma 0}|<\delta_0\}$, but equal to 1 near $\{p_{\gamma 0}=0\}$.
Then one can use the uncertainty principle~\cite[Section~4.2]{w-z}
to show that if $\chi_2$ is
supported inside $\{\chi_0=1\}$, but $\chi_2=1$ near $(x_0,0)$, then
for each $v\in L^2$,
$$
\begin{gathered}
((H_{\Real p_{\gamma, 0}}G_1)^w\chi_1^w v,\chi_1^w v)
\leq (-C^{-1}\tilde h+O(\tilde h^2))\|\chi_2^w v\|^2
+O(\log(1/h))\|(1-\chi_2^w)\chi_1^w v\|^2\\
-C_0C^{-1}\log(1/h)\|(1-\chi_2^w)\chi_1^w v\|^2
+O(C_0h\log(1/h))\|\chi_1^w v\|^2
+O(h^\infty)\|v\|^2\\
\leq -(C^{-1}\tilde h-O(\tilde h^2+C_0 h\log(1/h)))\|\chi_2^w v\|^2\\
-(C_0 C^{-1}\log(1/h)-O(C_0h\log(1/h)+\log(1/h)))\|(1-\chi_2^w)\chi_1^w v\|^2
+O(h^\infty)\|v\|^2.
\end{gathered} 
$$
If we fix $C_0$ large enough and $\tilde h$ small enough and assume that $h$ small enough, then
$$
\begin{gathered}
((H_{\Real p_{\gamma, 0}}G_1)^w\chi_1^w v,\chi_1^w v)
\leq -C^{-1}\log(1/h)\|(1-\chi_2^w)\chi_1^w v\|^2
-C^{-1}\tilde h \|\chi_1^w v\|^2
+O(h^\infty)\|v\|^2.
\end{gathered}
$$ 
Next, we conjugate by exponential pseudodifferential weights. First of all, one can prove that
$$
\|G_1^w\|_{L^2\to L^2}\leq C\log(1/h);
$$
therefore,
$$
\|e^{sG_1^w}\|_{L^2\to L^2}\leq h^{-C|s|}.
$$
Let $\chi$ be supported in $\{\chi_1=1\}$, but $\chi=1$ near $\{p_{\gamma 0}=0\}$, and
$$
P_{\gamma,s}=e^{sG_1^w}P_\gamma e^{-sG_1^w},\
u_{\gamma,s}=e^{sG_1^w}\chi^w u_\gamma;
$$
then~\cite[Section~4.3]{w-z}
$$
P_{\gamma,s}=P_\gamma+ish(H_{p_{\gamma 0}}G_1)^w+O(s^2\tilde h h+sh^{3/2}\tilde h^{3/2}+h^2).
$$
Therefore, since $\Imag p_{\gamma 0}=0$ near $\supp\chi_2$,
$$
\begin{gathered}
\Imag(\tilde P_{\gamma,s}\chi_1^w u_{\gamma,s},\chi_1^w u_{\gamma,s})
=\Imag(\tilde P_{\gamma0}\chi_1^w u_{\gamma,s},\chi_1^w u_{\gamma,s})
+h\Real(V_1\chi_1^w u_{\gamma,s},\chi_1^w u_{\gamma,s})\\
+sh\Real((H_{\Real p_{\gamma,0}}G_1)^w\chi_1^w u_{\gamma,s},\chi_1^w u_{\gamma,s})
+O(s^2h\tilde h+sh^{3/2}\tilde h^{3/2}+h^2)\|\chi_1^w u_{\gamma,s}\|^2\\
\leq O(h)\|(1-\chi_2^w)\chi_1^w u_{\gamma,s}\|^2+h\|V_1\|_{L^\infty}\|\chi_1^w u_{\gamma,s}\|^2
-C^{-1}sh\log(1/h)\|(1-\chi_2^w)\chi_1^w u_{\gamma,s}\|^2\\
-C^{-1}sh\tilde h\|\chi_1^w u_{\gamma,s}\|^2
+O(s^2h\tilde h+sh^{3/2}\tilde h^{3/2}+h^2)\|\chi_1^w u_{\gamma,s}\|^2
+O(h^\infty)\|u_\gamma\|^2.
\end{gathered}
$$
Here $\tilde P_{\gamma 0}$ is the principal part of $\tilde P_\gamma$, as before.
If we choose $s$ small enough independently of $h$, then for small $h$,
$$
\begin{gathered}
\Imag(\tilde P_{\gamma,s}\chi_1^w u_{\gamma,s},\chi_1^w u_{\gamma,s})
\leq -C_1sh\log(1/h)\|(1-\chi_2^w)\chi_1^w u_{\gamma,s}\|^2\\
-h(C^{-1}s\tilde h-\|V_1\|_{L^\infty})\|\chi_1^w u_{\gamma,s}\|^2
+O(h^\infty)\|u_\gamma\|^2.
\end{gathered}
$$
Now, $\|V_1\|_{L^\infty}$ can be made very small by choosing $\tilde\mu$ and $\nu$ small enough.
Then, we get
$$
\|\chi^w u_\gamma\|^2\leq Ch^{-1-Cs}\|u_\gamma\|\cdot\|P_\gamma\chi^w u_\gamma\|
+O(h^\infty)\|u_\gamma\|^2.
$$
By proceeding as in the end of Proposition~\ref{l:nontrapping}, we get~\eqref{e:trapping-est},
provided that $s$ is small enough. 
\end{proof}

\noindent\textbf{Acknowledgements.}
I would like to thank Maciej Zworski for suggesting the problem, lots
of helpful advice, and encouragement, and Kiril Datchev, Mihai Tohaneanu,
Daniel Tataru, and Tobias Schottdorf for some very helpful discussions.
I am also grateful for partial support from NSF grant~DMS-0654436.
Finally, I am especially thankful to an anonymous referee for many
suggestions to improve the manuscript.


\end{document}